\documentclass[letterpaper,12pt]{article}

\usepackage{graphicx}
\usepackage{appendix}
\usepackage{amssymb,amsmath}
\usepackage{subfigure}
\usepackage{amsthm}
\usepackage{hyperref}
\usepackage{color}
\usepackage{multirow,txfonts,verbatim}
\usepackage{mathrsfs}
\usepackage{booktabs}
\usepackage{bm}
\usepackage{tikz}
\usetikzlibrary{decorations.pathreplacing,calligraphy}

\usepackage{algorithm}
\usepackage{algpseudocode}
\usepackage[square,comma,sort&compress,numbers]{natbib}
\usepackage{enumerate}
\usepackage[shortlabels]{enumitem}

\usepackage{float}

\newtheoremstyle{plainNoItalics}{}{}{\normalfont}{}{\bfseries}{.}{ }{}
\theoremstyle{plain}

\theoremstyle{plainNoItalics}

\newtheorem{example}{Example}[section]

\newtheorem{criterion}{Criterion}
\theoremstyle{definition}\newtheorem{proposition}{Proposition}
\theoremstyle{definition}
\theoremstyle{definition}
\theoremstyle{definition}\newtheorem{theorem}{Theorem}
\theoremstyle{definition}\newtheorem{lemma}{Lemma}
\theoremstyle{definition}\newtheorem{rem}{Remark}
\numberwithin{equation}{section}
\numberwithin{theorem}{section}
\numberwithin{lemma}{section}
\numberwithin{figure}{section}
\numberwithin{example}{section}
\numberwithin{rem}{section}
\numberwithin{assumption}{section}

\usepackage{lineno}
\usepackage{xcolor}
\usepackage{natbib}

\newcommand{\rd}{{\mathrm{d}}}

\newcommand{\bv}{{\bm{v}}}
\newcommand{\bx}{{\bm{x}}}
\newcommand{\bs}{{\bm{s}}}
\newcommand{\bn}{{\bm{n}}}

\newcommand{\paran}[1]{\left( #1 \right)}

\newcommand{\SqBrack}[1]{\left[ #1 \right]}
\newcommand{\real}{\mathbb{R}}

\definecolor{ChhaiA}{rgb}{0.9, 0.9, 0.9}
\definecolor{ChhaiB}{rgb}{0.8, 0.8, 0.8}
\definecolor{ChhaiC}{rgb}{0.7, 0.7, 0.7}
\definecolor{ChhaiD}{rgb}{0.6, 0.6, 0.6}
\definecolor{ChhaiE}{rgb}{0.5, 0.5, 0.5}
\definecolor{ChhaiF}{rgb}{0.4, 0.4, 0.4}
\definecolor{ChhaiG}{rgb}{0.3, 0.3, 0.3}
\definecolor{ChhaiH}{rgb}{0.2, 0.2, 0.2}
\usepackage{tikzcd2}
\graphicspath{{figures/}}

\title{Physics-informed machine learning for reconstruction of dynamical systems with invariant measure score matching}
\author{
Yongsheng Chen
\thanks{School of Mathematical Sciences, Zhejiang University, Hangzhou, 310027, China. {\tt 22035024@zju.edu.cn}}
\and
Suddhasattwa Das
\thanks{Department of Mathematics and Statistics, Texas Tech University, Lubbock, TX, 70409, USA.  {\tt suddas@ttu.edu}.}
\and
Wei Guo
\thanks{Department of Mathematics and Statistics, Texas Tech University, Lubbock, TX, 70409, USA. 
{\tt weimath.guo@ttu.edu}. }
\and
Xinghui Zhong
\thanks{Corresponding author. School of Mathematical Sciences, Zhejiang University, Hangzhou, 310027, China. {\tt zhongxh@zju.edu.cn}}
}

\begin{document}

\maketitle
\begin{abstract}
In this paper, we develop a novel mesh-free framework, termed physics‑informed neural networks with invariant measure score matching (PINN‑IMSM), for reconstructing dynamical systems from unlabeled point‑cloud data that captures the system's invariant measure. The invariant density  satisfies the steady-state Fokker-Planck (FP) equation. We reformulate this equation in terms of its score function (the gradient of the log-density), which is estimated directly from data via denoising score matching, thereby bypassing explicit density estimation. This learned score is then embedded into a physics‑informed neural network (PINN) to reconstruct the drift velocity field under the resulting score-based FP equation. The mesh‑free nature of PINN allows the framework to scale efficiently to higher dimensions, avoiding the curse of dimensionality inherent in mesh‑based methods. To address the ill-posed nature of high-dimensional inverse problems, we further recast the problem as a PDE‑constrained optimization that seeks the minimal‑energy velocity field. Under suitable conditions, we prove that this problem admits a unique solution that depends continuously on the score function. The constrained formulation is solved using a stochastic augmented Lagrangian method.  Numerical experiments on representative dynamical systems, including the Van der Pol oscillator, an active swimmer in an anharmonic trap, chaotic Lorenz-63 and Lorenz-96 systems, demonstrate that PINN-IMSM accurately recovers invariant measures and reconstructs faithful dynamical behavior, successfully handling problems in up to five dimensions. 

\end{abstract}
\section{Introduction}

Reconstructing dynamical systems from observational data is a fundamental task across science and engineering, central to understanding and predicting complex phenomena from chaotic attractors and turbulent flows to pattern formation and neural activity.  The problem involves finding a mathematical model, often parameterized in a trainable manner, such that its simulated behavior aligns with the observed data.
Recent data-driven machine learning approaches have significantly advanced this field, providing flexible frameworks for systems with incomplete or noisy observations. Neural ordinary differential equations (ODEs) \cite{chen2018neural,portwood2019turbulence,linot2022data,linot2023stabilized} and sparse identification of nonlinear dynamics (SINDy) \cite{brunton2016discovering,messenger2021weak,fasel2022ensemble} have proven effective for dynamical system identification. However, these methods adopt a Lagrangian perspective and rely heavily on accurate trajectory data to estimate derivatives, limiting their applicability when temporal information is sparse or unavailable. A scenario for recovering unknown dynamical systems from data with missing time labels was considered in \cite{zeng2023reconstruction}, where the data are assumed to lie on a one-dimensional manifold. Their method minimizes the sliced-Wasserstein distance between samples generated by neural ODEs and training data, employing a trajectory segmentation technique to facilitate training. Nevertheless, the approach remains Lagrangian and is limited to lower-dimensional manifolds.

In this paper, we focus on the scenario where observational data are provided as a point cloud consisting of samples from an underlying dynamical system,  but lack reliable temporal labels due to slow or unrecorded sampling. Such data frequently appear in fields including molecular dynamics, biology, healthcare, medical research, and numerical weather prediction.  For dynamical systems, the distribution of states converges over time to a time‑independent  limit known as \emph{invariant measure}. 
The observational data, though devoid of explicit timing, provides an empirical approximation of this invariant measure, which in turn encodes key information about the system’s dynamic law.  Classical Lagrangian methods, which rely on time‑labeled trajectories to estimate derivatives, are ineffective in this label‑deficient setting.


To address the challenge, an Eulerian approach has been developed in \cite{yang2023optimal,botvinick2023learning} to learn the dynamics from such data.  This approach formulates the reconstruction as a partial differential equation (PDE) inverse problem, which infers the unknown velocity field directly from  data by imposing the steady‑state Fokker–Planck (FP) equation as a physical constraint, thereby circumventing the need for time‑labeled trajectories.
The method proceeds in three steps. First, the invariant density is approximated empirically by binning the data into a histogram $\rho^\star$. This invariant density is known to satisfy the steady-state FP equation. Second, a finite volume scheme is employed to compute the  solution $\rho(v(\theta))$ of the FP equation, for a  velocity field $v(\theta)$, which  is parameterized by a neural network with trainable parameters $\theta$. Third, $v(\theta)$ is optimized using a gradient-based method to minimize the loss function $\mathcal{J}(\theta):=\mathcal{D}(\rho(v(\theta)), \rho^\star)$, where $\mathcal{D}$ denotes a metric or divergence between probability measures.
This framework  has been validated on several benchmark examples. However, because it relies on a mesh-based PDE solver,  it suffers from  the curse of the dimensionality and is practically limited to problems of low dimensions $(d\le3)$. 



The practical scalability of the aforementioned mesh‑based Eulerian method is limited by their reliance on mesh-based discretization and the need to repeatedly solve the governing PDE, which becomes computationally prohibitive in higher dimensions.
This limitation has motivated the development of mesh-free learning paradigms that encode  the governing PDE structure directly directly within neural network parameterizations while maintaining differentiability for gradient-based optimization. Physics-informed neural networks (PINNs) \cite{raissi2019physics,karniadakis2021physics,wang2022and,jagtap2020conservative,stock2023bayesian} provide one such 
framework by embedding the equations as soft or hard constraints in the training loss, enabling parameter/field identification without an explicit mesh-based PDE solver. Another direction is to use deep operator networks, such as DeepONet \cite{lu2021learning}, Fourier neural operators \cite{li2021fourier}, and graph neural operators \cite{anandkumar2020neural}, which learn mappings between function spaces to recover unknown parameters from observational data. While powerful, these operator methods typically demand large volumes of high‑quality training data \cite{de2022cost,tanyu2023deep}. Recent works that incorporate PDE constraints into such operator‑learning frameworks, aim to reduce data requirements and enhance generalization capabilities \cite{li2024physics,rosofsky2023applications}. 



Motivated by the limitations of existing methods, we propose a novel mesh-free framework termed {\em{PINN with invariant measure score matching}} (PINN-IMSM), designed to reconstruct dynamical systems from data without explicit temporal labels. A key step is to reformulate the stationary FP equation in terms of the score function, i.e., the gradient of the log-density of the invariant measure, yielding a score‑based stationary FP equation. PINN‑IMSM first uses  denoising score matching \cite{hyvarinen2005estimation,song2020sliced} to learn this score function directly from unlabeled observational data, avoiding explicit density estimation. The learned score is subsequently embedded into a PINN to reconstruct the underlying velocity field by enforcing the physics constraint derived from the score-based FP equation.

To overcome the ill-posedness challenge in the standard PINN approach,  we further formulate the reconstruction as a PDE‑constrained optimization problem, seeking the minimal‑energy velocity field that satisfies the FP equation. We prove that, under suitable conditions, this PDE-constrained optimization problem admits a unique solution, and that the solution is continuous with respect to the score function.  This constrained formulation is solved efficiently via the stochastic augmented Lagrangian method \cite{dener2020training}. Numerical experiments demonstrate that the proposed PINN-IMSM method achieves accurate and stable reconstruction of dynamical systems from observational data, successfully handling problems in dimensions up to five, thereby significantly surpassing the scalability  of existing mesh-based and trajectory-based methods.

The remainder of this paper is organized as follows. In Section \ref{sec:back}, we introduce background on stochastic differential equations,  invariant measures, the FP equations,  and the score estimation and PINNs.  Section \ref{sec:method} is devoted to the details of the proposed PINN-IMSM framework. Section \ref{sec:num} presents numerical results validating the method's performance and robustness. Conclusions and directions of future work are discussed in Section \ref{sec:con}.

\section{Background}\label{sec:back}
\subsection{Stochastic differential equations, invariant measures, and Fokker-Planck equations}
Consider the stochastic differential equation (SDE) of the form
\begin{equation} \label{eq:sde}
\mathrm{d}  X_t=\bv\left(X_t\right) \mathrm{d} t+\Gamma\left(X_t\right) \mathrm{d}  W_t,
\end{equation}
where $X_t$ is a $d$-dimensional stochastic process with the state space $\Omega \subset \mathbb{R}^d$ being a bounded Lipschitz domain, $\bv: \Omega \to \mathbb{R}^d$ is the drift velocity, $W_t$ is a standard $m$-dimensional Brownian motion, and $\Gamma: \Omega \to \mathbb{R}^{d \times m}$  is the $d\times m$-dimensional diffusion coefficient matrix. We assume $v$ and $\Gamma$ are Lipschitz continuous. 
The solution to \eqref{eq:sde} is given by the integral equation:
\[X_t = X_0 + \int_0^t \bv \paran{ X_s } \mathrm{d} s + \int_0^t \Gamma \paran{X_s} \mathrm{d} W_s .\]

 If the initial condition $X_0$  is random with a probability density $\rho_0$, then the SDE induces a flow of probability densities $\rho(\bx, t)$. The evolution of this density is governed by the FP equation \cite{huang2015steady, Ovchinnikov2016supersymm, baxendale2007stochastic}
\begin{equation} \label{eq:fp-pde}
     \frac{\partial \rho(\bx,t)}{\partial t} = - \nabla  \cdot (\rho(\bx,t) \bv(\bx)) + \nabla \cdot \Big( \nabla \cdot (\Sigma(\bx) \rho(\bx,t))\Big)\triangleq \mathcal{F}\rho(\bx,t), 
\end{equation}
with the initial condition $\rho(\bx,0)=\rho_0(\bx)$ and the constraints:
\begin{align}\label{cond_u}
    \rho(\bx,t) \geq 0, \qquad \int_{ \Omega } \rho(\bx,t) \mathrm{d} \bx = 1.
\end{align}
Here, $\mathcal{F}$ is the FP operator, and $\Sigma$ is the diffusion coefficient matrix with its entries given by
\[\begin{split}
  & \Sigma_{i,j}(\bx) :=\dfrac{1}{2} \SqBrack{\Gamma(\bx) \Gamma(\bx)^T}_{i,j} = \dfrac{1}{2}\left\langle \Gamma_i(\bx), \Gamma_j(\bx) \right\rangle_{ \real^{m} },
\end{split}\]
where $\Gamma_i(\bx)$ denotes the $i$-th row of the matrix $\Gamma(\bx)$. The matrix $\Sigma$ is symmetric and positive semi-definite. To ensure probability conservation within $\Omega$, 
FP equation \eqref{eq:fp-pde} is supplemented with homogeneous Neumann (reflecting) boundary conditions:
\begin{align}
    \label{eq:bc}
    \left( -\bv(\bx)\rho(\bx,t) + \nabla \cdot (\Sigma(\bx)\rho(\bx,t)) \right) \cdot \mathbf{n}(\bx) = 0, \quad \bx \in \partial\Omega,
\end{align}
where $\mathbf{n}(\bx)$ is the outward unit normal to $\partial\Omega$.

A density $\rho(\bx)$ is stationary or invariant if it remains unchanged under this transformation, i.e., $\mathcal{F} \rho = 0$ subject to the constraints \eqref{cond_u}. This represents a statistical equilibrium that balances  the deterministic drift and stochastic diffusion. Furthermore, an invariant density $\rho$ defines  an invariant probability measure $\mu$ on the phase space by $\mu(U) = \int_U \rho(\bx) \mathrm{d} \bx$ for any measurable set $U\subset\Omega$. This measure is invariant in the sense that if $X_0$ is distributed according to $\mu$, then $X_t$ is also distributed according to $\mu$ for all $t > 0$. Thus, the probability of  the process being in any set $U$ remains constant over time. The determination of these stationary densities is of significant practical importance, as they directly identify high- and low-probability regions in the phase space.

In this paper,  we focus on the scenario of constant, isotropic diffusion, where the diffusion matrix takes the form $\Sigma(\bx)=DI$, where $D$ is a positive constant known as the diffusion coefficient, and $I$ is the identity matrix. 
Consequently, the general SDE \eqref{eq:sde} simplifies to
\begin{align}
    \label{eq:sde_D}
    \rd X_t = \bv(X_t)\,\rd t + \sqrt{2D}\,\rd W_t,
\end{align}
which represents a dynamical system driven by isotropic white noise of strength $D$. In the zero-noise limit $D\to 0^+$, the system reduces to the continuity equation modeling the probability flow of the ODE given by $\Dot{\bx} = \bv(\bx)$. For systems exhibiting stochastic stability \cite{StochStable, Young1986StochHyp, CowiesonYoung2005}, the invariant measure depends continuously on the noise level. This property, while difficult to establish theoretically, is widely observed empirically. For such systems, the constant isotropic diffusion model ensures the stationary distribution remains close to that of the noise-free counterpart.

The corresponding FP equation \eqref{eq:fp-pde} then simplifies to
\begin{equation}\label{eq:sp-pde}
     \frac{\partial \rho(\bx,t)}{\partial t} = - \nabla  \cdot (\rho(\bx,t) \bv(\bx)) + D\nabla^2 \rho(\bx,t), \quad \bx\in\Omega,
\end{equation}
associated with the boundary condition
\begin{equation}\label{eq:bc:simple}
(D\nabla\rho-\rho\,\bv)\cdot \bn=0 \quad \text{on }\partial\Omega.
\end{equation}
Under certain conditions \cite{huang2015steady}, the invariant density $\rho(\bx)$ of \eqref{eq:sp-pde} exists and satisfies
\begin{equation}\label{eq:steady-pde}
     \nabla  \cdot (\rho(\bx) \bv(\bx)) = D\nabla^2 \rho(\bx),\quad \bx\in\Omega.
\end{equation}
This formulation implies that if the invariant density $\rho(\bx)$ can be estimated from data, the underlying velocity field $\bv$ can be reconstructed by solving the inverse problem defined by \eqref{eq:steady-pde}. Our proposed method is developed within this paradigm.

\subsection{Score estimation}\label{nscn}
The inverse problem posed by Equation \eqref{eq:steady-pde} requires an accurate estimate of the invariant density $\rho(\bx)$ from data. A critical first step is therefore to obtain a robust approximation of this invariant measure.
In \cite{botvinick2023learning}, this measure is approximated using the \emph{occupation measure}, which is the empirical distribution obtained by binning observed points into histogram. However, this binning approach is well-known to suffer from the curse of dimensionality, rendering it impractical for high-dimensional problems \cite{bellman1961adaptive}. To overcome this limitation, we employ \emph{score matching} techniques, which have demonstrated remarkable success in high-dimensional density estimation \cite{hyvarinen2005estimation,song2020sliced}.

For any continuously differentiable probability density function $\rho(\bx)$, the associated score function $\bs(\bx)$ is defined as 
\begin{equation}\label{eq:score-def}
    \bs(\bx)=\nabla_\bx \log \rho(\bx).
\end{equation}
Estimating the score function is often more tractable than direct density estimation, as it circumvents the computation of intractable normalizing constants \cite{song2020sliced,hyvarinen2005estimation}. Moreover,  the score function provides crucial information about the geometry of the underlying probability distribution and enables various downstream applications, including sample generation through Langevin Monte Carlo dynamics and gradient-based optimization procedures \cite{song2020score,ho2020denoising}.

A score network $\mathbf{s}_{\theta}(\bx)$, parameterized with trainable parameters $\theta$, is used to approximate the ground-truth score function $\nabla_{\bx} \log \rho(\bx)$. The natural objective is to minimize the expected squared error $\frac{1}{2}\mathbb{E}_{\rho(\bx)}||\mathbf{s}_{\theta}(\bx)-\nabla_x \log \rho(\bx)||^2_2$, which defines the classical score matching framework \cite{hyvarinen2005estimation}. However, this formulation cannot be applied directly because the true score function is unknown.

To overcome this challenge, numerous effective score matching variants have been developed in the literature. A widely used approach is denoising score matching \cite{vincent2011connection}, which has become a cornerstone technique in modern generative modeling. Given observed data $\bx \in \Omega$ from the unknown  density  $\rho(\bx)$, the main idea of the denoising score matching is to first perturb it with additive noise according to a conditional distribution $p_{\sigma}(\tilde{\bx}| \bx)$, where $\sigma$ controls the noise level. Score matching is then applied to estimate the score of the noise-perturbed distribution  $p_{\sigma}(\tilde{\bx}) \triangleq \int_{\Omega} p_{\sigma}(\tilde{\bx}| \bx)\rho(\bx)\mathrm{d} \bx$. As proven in \cite{vincent2011connection}, this is equivalent to minimizing the objective:
\begin{equation}\label{eq:denoise}
    \frac{1}{2}\mathbb{E}_{\rho(x)}\mathbb{E}_{p_\sigma(\tilde{x}\mid x)}\left[\|\mathbf{s}_{\theta}(\tilde{\bx})-\nabla_{\tilde{\bx}} \log p_{\sigma}(\tilde{\bx}| \bx)\|^2_2 \right],
\end{equation}
where the expectation over $\rho(\bx)$ is approximated using given data samples $\bx$. The theoretical study in  \cite{vincent2011connection} ensures that, under appropriate regularity conditions, the optimal network $\mathbf{s}_{\theta^*}(\bx)$ that minimizes this objective \eqref{eq:denoise}  satisfies 
\begin{align}
    \label{eq:scoreappro}
    \mathbf{s}_{\theta^*}(\bx) = \nabla_\bx \log p_{\sigma} (\bx), \quad \textnormal{almost surely}.
\end{align}
When the noise level $\sigma$ is sufficiently small such that $p_{\sigma} (\tilde{\bx}) \approx \rho(\tilde{\bx})$, the optimal score network yields a close approximation to the true score function:
$$\mathbf{s}_{\theta^*}(\bx) = \nabla_\bx \log p_{\sigma} (\bx) \approx \nabla_\bx \log \rho (\bx).$$


Despite its theoretical appeal, the single-scale denoising approach in \eqref{eq:denoise} often exhibits unsatisfactory performance in practice, particularly in low density regions where the signal-to-noise ratio is unfavorable \cite{song2019generative}. This issue is addressed by the multi-scale denoising approach \cite{song2019generative}, which perturbs data with multiple noise levels and trains a single conditional score network to estimate scores across all noise scales  simultaneously. This technique has since become a standard  in score-based generative modeling.

The multi-scale denoising approach involves selecting a geometric sequence of noise scales $\{\sigma_i\}_{i=1}^L$, such that
$$\frac{\sigma_1}{\sigma_2} =\cdots = \frac{\sigma_{L-1}}{\sigma_L} = \gamma >1,$$ 
where $\gamma$ is a fixed ratio. To determine the parameters of this sequence ($\sigma_1$, $\gamma$, and $L$), we follow the systematic approach proposed in \cite{song2020improved}, which is based on two  criteria:

\begin{criterion}
Set the largest noise scale $\sigma_1$ to the maximum pairwise Euclidean distance among all training data points. This ensures the Gaussian noise distribution at the largest scale adequately covers the entire support of the data.
\end{criterion}

\begin{criterion}
Choose the  ratio $\gamma$  such that it satisfies
$$\Phi(\sqrt{2d}(\gamma-1)+3\gamma)-\Phi(\sqrt{2d}(\gamma-1)-3\gamma)\approx 0.5,$$ 
where $\Phi$ denotes the standard normal cumulative distribution function. This ensures sufficient overlap between consecutive noise levels for effective training.
\end{criterion}
\noindent The number of scales $L$ is then chosen such that the smallest noise scale $\sigma_L$ is above a threshold (typically $\sigma_L \approx 0.01$), resulting in a well-conditioned noise sequence $\{\sigma_i\}_{i=1}^L$.

At each scale $\sigma_i$, the data is perturbed using a Gaussian distribution $p_{\sigma_i}(\tilde{\bx}\mid \bx) = \mathcal{N}(\tilde{\bx}\mid\bx, \sigma_i^2 \mathbf{I})$, yielding the noise-perturbed distribution $p_{\sigma_i}(\tilde{\bx}) \triangleq \int p_{\sigma_i}(\tilde{\bx}\mid\bx)\rho(\bx)\mathrm{d} \bx$. A single conditional score network $\mathbf{s}_{\theta}(\bx, \sigma)$ is then trained to jointly approximate the
scores of all perturbed distributions, i.e., $\mathbf{s}_{\theta}(\bx, \sigma_i) \approx \nabla_x \log p_{\sigma_i} (\bx)$ for all  $i$. For a fixed noise level $\sigma_i$, the denoising score matching objective derived from \eqref{eq:denoise} is
\begin{equation*}
\frac{1}{2 } \mathbb{E}_{\rho(\bx)} \mathbb{E}_{p_{\sigma_i}(\tilde{\bx} \mid \bx)}\left[\Big\| \mathbf{s}_{\boldsymbol{\theta}}\left(\tilde{\bx}, \sigma_i\right) + \frac{\tilde{\bx}-\bx}{\sigma_i^2}\Big\|_2^2\right].
\end{equation*}
These individual objectives are combined into a unified loss function for all noise scales:
\begin{equation}\label{eq:condition_unify}
\mathcal{L}_s = \frac{1}{2 L} \sum_{i=1}^L \lambda(\sigma_i)\,\mathbb{E}_{\rho(\bx)} \mathbb{E}_{p_{\sigma_i}(\tilde{\bx} \mid \bx)}\left[\Big\| \mathbf{s}_{\boldsymbol{\theta}}\left(\tilde{\bx}, \sigma_i\right)+\frac{\tilde{\bx}-\bx}{\sigma_i^2}\Big\|_2^2\right],
\end{equation}
where $\lambda(\sigma_i)>0$ are scaling coefficients.  Empirically, it is observed that $\|\mathbf{s}_{\theta}(\bx, \sigma)\|_2 \propto 1/\sigma$ for optimally trained networks \cite{song2019generative}, which motivates the choice $\lambda(\sigma_i) = \sigma_i^2$ to balance the loss across scales. Substituting this weighting yields the final objective: 
\begin{equation}\label{eq:condition}
\mathcal{L}_s = \frac{1}{2 L} \sum_{i=1}^L \mathbb{E}_{\rho(\bx)} \mathbb{E}_{p_{\sigma_i}(\tilde{\bx} \mid \bx)}\left[\Big\|\sigma_i \mathbf{s}_{\boldsymbol{\theta}}\left(\tilde{\bx}, \sigma_i\right)+\frac{\tilde{\bx}-\bx}{\sigma_i}\Big\|_2^2\right].
\end{equation}

To facilitate learning across different noise scales and ensure numerical stability, we employ a parameterization of the conditional score network: $\mathbf{s}_{\theta}(\bx,\sigma)=\mathbf{s}_{\theta}(\bx)/\sigma$, where $\mathbf{s}_{\theta}(\bx)$ is an unconditional score network. This reparameterization mitigates the numerical instabilities that arise as $\sigma \to 0$ \cite{song2020improved}, while the lower bound $\sigma_L\approx 0.01$ provides additional regularization.

\subsection{Physics-informed neural networks}
This section reviews the framework of PINNs \cite{raissi2019physics} for solving PDEs. PINNs provide a powerful, mesh-free alternative to traditional numerical methods, particularly for high-dimensional problems and complex geometries.

Consider a PDE of the general form
\begin{equation}\label{eq:pde}
     \mathcal{N}\mathbf{u}(\bx) = \mathbf{f}(\bx), \quad\bx\in \Omega,
\end{equation}
where $\mathcal{N}$ is a nonlinear differential operator, $\mathbf{u}: \mathbb{R}^d \to \mathbb{R}^{m}$ is the unknown solution field, $\mathbf{f}: \mathbb{R}^d \to \mathbb{R}^m$ is a prescribed source function, and $\Omega \subset \mathbb{R}^d$ is the computational domain. Traditional numerical methods, such as finite difference and finite element approaches, rely on spatial discretizations that can become computationally expensive for high-dimensional problems due to the curse of dimensionality \cite{quarteroni2008numerical,brenner2008mathematical}.  In contrast, PINNs \cite{raissi2019physics} employ a neural network $\hat{\mathbf{u}}(\bx;\theta)$ to  approximate the solution directly, where $\theta$ denotes the collection of all trainable network parameters. The key innovation of PINNs is the incorporation of physical constraints into the training process  via automatic differentiation \cite{baydin2017automatic}, enabling the network to learn solutions that satisfy both the governing equations and available measurement data. 

The framework utilizes two distinct point sets: observational data  $\{(\bm y_i,\mathbf{u}_i)\}_{i=1}^{N_u}$ (which may be sparse or absent) and collocation points $\{\bx_j\}_{j=1}^{N_f}$ sampled throughout the domain  $\Omega$. 
The training process  minimizes a composite physics-informed loss function:
\begin{equation}\label{eq:pinn-loss}
    \mathcal{L} = \omega_1 \mathcal{L}_{\text{data}}+  \omega_2 \mathcal{L}_{\text{PDE}},
\end{equation}
with the data fidelity term 
\[\mathcal{L}_{\text{data}} = \frac{1}{N_u}\sum_{i=1}^{N_u}\|\hat{\mathbf{u}}(\bm y_i;\theta)-\mathbf{u}_i\|^2_2,\]
and the physics-based residual term
$$\mathcal{L}_{\text{PDE}} = \frac{1}{N_f}\sum_{j=1}^{N_f}\| \mathcal{N}\hat{\mathbf{u}}(\bx_j;\theta)-\mathbf{f}(\bx_j)\|^2_2.$$
The weighting coefficients $\omega_1$ and $\omega_2$  balance contributions of  data fitting and  physics enforcement.

The optimization of the loss function \eqref{eq:pinn-loss} is typically performed using gradient-based algorithms such as the Adam optimizer \cite{kingma2014adam} or L-BFGS \cite{liu1989limited}, leveraging the automatic differentiation capabilities of modern deep learning frameworks. The trained network $\hat{\mathbf{u}}(\bx;\theta^*)$ provides a mesh-free approximation of the  solution that can be evaluated at arbitrary points within the domain. Beyond forward problems, PINNs have demonstrated remarkable effectiveness in addressing inverse problems, where unknown parameters or functions in the governing equations are inferred from limited data \cite{kadeethum2020physics,depina2022application,guo2023high}. In this work, we leverage this capability  to address the challenging task of reconstructing unknown velocity \(\bv\) in \eqref{eq:steady-pde} from observed trajectory data.


\section{PINN-IMSM: PINN with invariant measure score matching}\label{sec:method}
This section introduce PINN-IMSM, a novel mesh-free machine learning framework that combines PINN with invariant measure score matching to reconstruct dynamical systems from noisy trajectory observations.
The key innovation of PINN-IMSM lies in its ability to learn velocity fields without requiring explicit time labels, making it particularly suitable for real-world applications where temporal information is incomplete or unavailable. Our approach addresses three fundamental challenges in dynamical system reconstruction: (i) handling intrinsic noise in trajectory data, (ii) circumventing the curse of dimensionality inherent in traditional mesh-based methods, and (iii) ensuring well-posed reconstruction in high-dimensional spaces.

\subsection{Problem reformulation and method overview}
The theoretical foundation of PINN-IMSM rests on utilizing the steady-state FP equation \eqref{eq:steady-pde} as a forward model to characterize the invariant measure generated by the observed trajectories. 

A central component of PINN-IMSM is the use of score functions rather than density functions for invariant measure estimation. We reformulate the steady-state FP equation \eqref{eq:steady-pde} by substituting the score definition \eqref{eq:score-def}, yielding
\begin{equation}\label{eq:score-pde}
\bs(\bx) \cdot \bv(\bx) + \nabla \cdot \bv(\bx) = D \left( |\bs(\bx)|^2 + \nabla \cdot \bs(\bx) \right),\quad \bx\in\Omega.
\end{equation}
In this formulation, the diffusion term serves a dual purpose: it accommodates intrinsic noise while preventing overfitting to noise components in the velocity field $\bv(\bx)$, and simultaneously provides natural regularization. This approach distinguishes PINN-IMSM from existing mesh-based finite volume approaches \cite{botvinick2023learning} and enables the effective capture of complex dynamical behaviors in high-dimensional spaces. 

Given trajectory measurements $\{\widetilde{X}_i\}_{i=1}^{N}$,  interpreted as samples from stochastic dynamical systems \eqref{eq:sde}, PINN-IMSM reconstructs the associated velocity field $\bv(\bx)$ through a two-stage process:
\begin{itemize}
    \item {\bf Stage 1:}  Train a score network $\bs_{\theta_1}(\bx, \sigma_L)$ using denoising score matching to approximate the score function $\bs(\bx)$.
    \item {\bf Stage 2:} Reconstruct the velocity field $\bv(\bx)$ through a physics-informed approach based on Equation \eqref{eq:score-pde}.
\end{itemize}

\subsection{Score matching for invariant measures}
\label{sec:m-score}
The first stage of our methodology employs a noise-conditional score network $\bs_{\theta_1}(\bx, \sigma_L)$ to approximate the score function $\bs(\bx) = \nabla_{\bx} \log \rho(\bx)$, where $\rho(\bx)$ is the density of the invariant measure and $\theta_1$ denotes the trainable parameters.  The network is trained using the multi-scale denoising score matching framework detailed in Section \ref{nscn} by minimizing the loss function \eqref{eq:condition} across a well-conditioned geometric sequence of noise levels $\{\sigma_i\}_{i=1}^L$.

We treat the trajectory measurements $\{\widetilde{X}_i\}_{i=1}^{N}$ as time-agnostic samples from the invariant measure, forming the empirical distribution $\widehat{\rho}_N \coloneqq \tfrac{1}{N}\sum_{i=1}^{N}\delta_{\widetilde{X}_i}$ on $\Omega$, where $\delta_{\widetilde{X}_i}$ denotes the Dirac mass at $\widetilde{X}_i$. 
During training, we repeatedly sample data points $\bx \sim \widehat{\rho}_N$, 
pick a noise level $\sigma_i$ from the fixed noise sequence $\{\sigma_i\}_{i=1}^L$, 
and draw a standard Gaussian perturbation $\boldsymbol{\xi} \sim \mathcal{N}(\mathbf{0},I_d)$. 
For each pair $(\bx,\sigma_i)$ we then form the perturbed sample
\[
\tilde{\bx} = \bx + \sigma_i \boldsymbol{\xi},
\]
and assemble mini-batches of such triplets $(\bx,\tilde{\bx},\sigma_i)$ to evaluate and minimize 
the multi-scale denoising score matching loss~\eqref{eq:condition}.

The key insight is the selection of a sufficiently small smallest noise level $\sigma_L$ such that the perturbed density approximates the density of the invariant measure, i.e., $p_{\sigma_L}(\bx) \approx \rho(\bx)$. 
 This choice, combined with the theoretical guarantees of  score matching \eqref{eq:scoreappro}, leads to the following approximation:
\begin{equation}
\bs_{\theta_1^*}(\bx, \sigma_L) \approx \nabla_x \log p_{\sigma_L} (\bx) \approx \nabla_x \log \rho (\bx) = \bs(\bx),
\end{equation}
which provides PINN-IMSM with a robust estimate of the score function \(\bs(\bx)\) for the underlying dynamical system. To ensure compatibility with automatic differentiation in the subsequent physics-informed learning stage, the score network utilizes smooth activation functions, specifically the swish activation \cite{eger2018time}. 

With the reliable estimate $\bs_{\theta_1^*}(\bx,\sigma_L)\approx \bs(\bx)$ obtained, we proceed to the second stage, which leverages the score-based reformulation \eqref{eq:score-pde} to recover the underlying velocity field.

\subsection{PINN-based velocity field reconstruction}
\label{sec:m-pinn}
The second stage of PINN-IMSM constructs a neural network $\bv_{\theta_2}(\bx)$ parameterized by $\theta_2$ to reconstruct the velocity field $\bv(\bx)$ using the score estimate obtained in the first stage. 

\subsubsection{The ill-posedness challenge in the standard PINN approach}

We begin by reformulating Equation \eqref{eq:score-pde} in compact operator form:
\begin{equation}\label{eq:compact-op}
\widetilde{\mathcal{N}}(\bs(\bx), \bv(\bx)) = \bs(\bx) \cdot \bv(\bx) + \nabla \cdot \bv(\bx) - D \left( |\bs(\bx)|^2 + \nabla \cdot \bs(\bx)\right)=0.
\end{equation}

The standard PINN approach trains the neural network $\bv_{\theta_2}(\bx)$   by minimizing  the physics-informed loss function over  the trajectory measurements $\{\widetilde{X}_i\}_{i=1}^{N}$:
\begin{equation}\label{eq:dp-loss}
    \mathcal{L}_{PDE} = \frac{1}{N}\sum_{i=1}^{N}\| \widetilde{\mathcal{N}}(\bs_{\theta_1^*}(\widetilde{X}_i, \sigma_L) ,\bv(\widetilde{X}_i)) \|^2_2,
\end{equation}
where  $\bs_{\theta_1^*}(\widetilde{X}_i,\sigma_L)$ is  the pre-trained score estimate. The optimized network $\bv_{\theta^*_2}(\bx)$ then approximates the velocity field $\bv(\bx)$, as illustrated in Figure \ref{fig:frame}.
 \begin{figure}[!t]
\centerline{\includegraphics[width=0.95\textwidth]{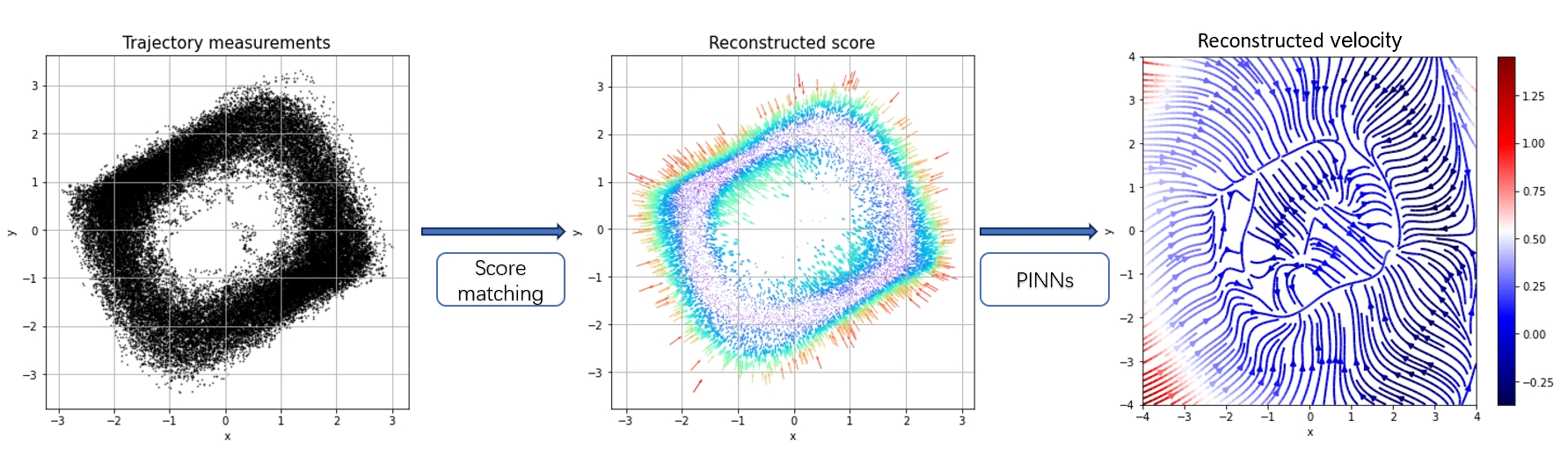}}
\caption{Velocity field reconstruction from trajectory data without explicit time labels. From trajectory measurements (left), we first reconstruct the score function via denoising score matching (middle). This reconstructed score is then integrated into the PINN framework to infer the velocity field across the state space (right).}
      \label{fig:frame}
\end{figure} 

Despite this straightforward formulation, the standard PINN approach faces several challenges. The primary issue is ill-posedness: Equation \eqref{eq:score-pde} may admit multiple velocity field solutions $\bv(\bx)$ for a given score function $\bs(\bx)$, particularly in high-dimensional settings. Furthermore, the absence of explicit data-fitting terms in the loss function complicates optimization, making minimization of \eqref{eq:dp-loss} particularly difficult \cite{raissi2020hidden,goodfellow2016deep}. These limitations motivate our reformulation as a constrained optimization problem to ensure well-posedness and stability.
\subsubsection{Ensuring well-posedness: PDE-constrained optimization approach}
\label{sec:m-constrained}
To address the ill-posedness of the velocity reconstruction problem identified in the previous section, we reformulate it as the following PDE-constrained optimization problem:
\begin{equation}\label{eq:true-prob}
\begin{aligned}
&\underset{\bv}{\text{minimize}} \quad  ||\bv(\bx)||^2 \\
&\text{{subject to}} \quad   \widetilde{\mathcal{N}}(\mathbf{s}(\bx),\bv(\bx))=0,
\end{aligned}
\end{equation}
where $||\bv(\bx)||^2=\int_{\Omega}|\bv(\bx)|^2 \, d\bx$. This formulation selects the minimal-energy solution among all feasible solutions consistent with the score-based FP constraint.


The well-posedness of this constrained optimization approach is established by the following theorem. The required preliminaries including definitions of the relevant function spaces and norms, together with a complete proof are provided in Appendix \ref{sec:appendix}.

\begin{theorem}[Well-posedness of \eqref{eq:true-prob}]\label{them:unique} Assume $\bs\in W^{1,\infty}(\Omega;\mathbb{R}^d)$
and that $q_{\bs}=|\bs|^2+\nabla\cdot\bs\ge0$ almost everywhere in $\Omega$. Then the constrained optimization problem \eqref{eq:true-prob} admits a \emph{unique} minimizer $\bv^*$. Moreover, the solution map $\bs\mapsto \bv^*$ is locally Lipschitz with respect to the $W^{1,\infty}$-metric on $\bs$ and the $L^2$-metric on $\bv$, i.e., there exist a neighborhood $U$ of $\bs$ and a constant $C>0$ such that, for any $\widetilde{\bs}\in U$, the corresponding minimizer $\widetilde{\bv}^*$ satisfies \[ \|\bv^*-\widetilde{\bv}^*\|_{L^2(\Omega)} \;\le\; C\,\|\bs-\widetilde{\bs}\|_{W^{1,\infty}(\Omega)}. \] 
\end{theorem}
\begin{rem}
 The condition $q_{\bs}\ge0$  is sufficient is sufficient for the theorem. The conclusion
remains valid under the weaker condition $\|q_{\bs}^-\|_{L^\infty(\Omega)} < C_P^{-2}$, 
where $q_{\bs}^-(x)=\max\{-q_{\bs}(x),0\}$ and $C_P$ is the Poincar\'e constant for $H^1_\diamond(\Omega)$.
\end{rem}
This theorem guarantees that the minimal-energy solution exists, is unique, and depends continuously on the score function $\bs$. The Lipschitz continuity ensures that the reconstruction of the velocity field is stable with respect to perturbations in the score function. Specifically, if we approximate the true score $\bs$ by an estimate $\widetilde{\bs}$ obtained from data (e.g., via score matching), the error in the reconstructed velocity field $\|\bv^*-\widetilde{\bv}^*\|_{L^2}$ is linearly controlled by the estimation error $\|\bs-\widetilde{\bs}\|_{W^{1,\infty}}$. This provides a theoretical guarantee that an accurate score estimate leads to an accurate reconstruction of the underlying dynamics, which is crucial when the score is learned from finite samples.

In practice, we work with the data-driven surrogate of \eqref{eq:true-prob} obtained by replacing the true score $\bs$ with its learned approximation $\bs_{\theta_1^*}$:
\begin{equation}\label{eq:sub-prob}
\begin{aligned}
&\underset{\bv}{\text{minimize}} \quad  \|\bv(\bx)\|^2 \\
&\text{{subject to}} \quad   \widetilde{\mathcal{N}}(\mathbf{s}_{\theta_1^*}(\bx),\bv(\bx))=0.
\end{aligned}
\end{equation}

Given that neural networks with smooth activation functions  yield $C^{m}$-smooth approximations, the learned score function $\bs_{\theta_1^*}(\bx) \in C^{m}(\Omega)$ satisfies the regularity assumptions of Theorem~\ref{them:unique}. Thus, the surrogate problem \eqref{eq:sub-prob} admits a unique solution $\widehat{\bv}$ and enjoys Lipschitz stability. In particular, an $\mathcal{O}(\varepsilon)$ error in $\bs_{\theta_1^*}$ (measured in $W^{1,\infty}$) leads to an $\mathcal{O}(\varepsilon)$ error in the reconstructed velocity field (measured in $L^2$). As learning improves $\bs_{\theta_1^*}$ and the minimal-energy constraint regularizes residual high-frequency noise, $\widehat{\bv}$ converges to the ideal solution $\bv^*$.

The theoretical foundation for our approach leverages the Universal Approximation Theorem \cite{hornik1989multilayer,chen1995universal}, ensuring that neural networks with sufficiently smooth activation functions can approximate any smooth function with arbitrary precision. Thus, solving problem \eqref{eq:sub-prob} using neural network architectures provides a robust computational approach to approximating the velocity field solution $\bv^*$.

\subsubsection{Numerical implementation: stochastic augmented Lagrangian method}
With the well-posedness of the constrained optimization problem \eqref{eq:sub-prob} established in Theorem~\ref{them:unique}, we now consider its numerical implementation using neural networks.
Several approaches have been proposed for constrained learning in physics-informed settings including soft penalty methods \cite{raissi2019physics}, penalty-based hard constraint formulations \cite{lu_physics-informed_2021},  adversarial training frameworks  \cite{zang2020weak}, and stochastic augmented Lagrangian methods \cite{dener2020training}. Among these, we adopt the stochastic augmented Lagrangian approach introduced in \cite{dener2020training}, originally developed  for physics-constrained training of encoder-decoder networks to approximate the Fokker-Planck-Landau collision operator. This approach builds upon the classical augmented Lagrangian framework,  originally known as the method of multipliers and first proposed by \cite{hestenes1969multiplier} and \cite{powell1969method} for nonlinear constrained optimization problems.

We adapt the stochastic augmented Lagrangian approach to our score-based PDE-constrained optimization formulation \eqref{eq:sub-prob}. Rather than approximating a known physical operator while preserving conservation laws, we aim to reconstruct an unknown minimum-energy velocity field from unlabeled trajectory data through a PDE constraint involving a separately learned score function.

Specifically, the augmented Lagrangian approach  reformulates the PDE-constrained optimization problem \eqref{eq:sub-prob} into a sequence of unconstrained optimization problems indexed by $k$:
\begin{equation}\label{eq:unconstrained}
\underset{\theta_2}{\text{minimize}} \quad \mathcal{L}(\theta_2,\lambda_k,\mu_k),
\end{equation}
with $\mathcal{L}$ being the augmented Lagrangian merit function:
\begin{equation}\label{eq:lag}
    \mathcal{L}(\theta_2,\lambda_k,\mu_k)=\|\bv_{\theta_2}\|^2+\lambda_k  \widetilde{\mathcal{N}}(\theta_2)+\frac{\mu_k}{2}\| \widetilde{\mathcal{N}}(\theta_2)\|^2,
\end{equation}
where $\lambda_k$ is the Lagrange multiplier and $\mu_k$ is the augmented Lagrangian penalty factor. 
Here, $\widetilde{\mathcal{N}}(\theta_2) := \widetilde{\mathcal{N}}(\bs_{\theta_1^*}, \bv_{\theta_2})$ denotes the PDE residual from \eqref{eq:compact-op}, evaluated using the score estimate $\bs_{\theta_1^*}$ obtained in the first stage and  the parameterized velocity field $\bv_{\theta_2}$.  We remark  that for the constrained formulation,   the original PINN loss \eqref{eq:dp-loss}  is no longer minimized directly. Instead, the physics residual appears as a constraint in \eqref{eq:sub-prob}, which is handled through the augmented Lagrangian framework in \eqref{eq:unconstrained}-\eqref{eq:lag}.

The stochastic augmented Lagrangian method employs a two-level nested loop structure: the inner loop uses stochastic gradient descent (SGD) to solve the unconstrained optimization problem \eqref{eq:unconstrained}, while the outer loop  updates Lagrange multipliers and penalty parameters in \eqref{eq:lag}.

\paragraph{Inner-loop optimization via stochastic gradient descent}
We  solve the unconstrained optimization problem \eqref{eq:unconstrained} using mini-batch SGD, adapting the approach from \cite{dener2020training} to our PDE-constrained setting.

Classical augmented Lagrangian methods solve the inner optimization problem \eqref{eq:unconstrained} to a gradient-norm tolerance $w_k$, i.e.,
\begin{equation}
\|\nabla_{\theta_2} \mathcal{L}(\theta_{2,k},\lambda_k,\mu_k)\|_2 \leq w_k,
\end{equation}
with  $w_k$ gradually tightened as the algorithm converges toward the constrained solution. However, this criterion is impractical in stochastic settings for two reasons: first, accurately estimating the optimality norm $\|\nabla_{\theta_2} {\mathcal{L}}\|_2$ is infeasible  with stochastic mini-batch gradients; second, gradient norms do not serve as meaningful termination criteria in machine learning, where prediction accuracy of the neural network is the primary objective.

Instead of checking convergence via gradient norms, the  SGD algorithm  is configured to iterate through  the entire training dataset once per outer-loop iteration.  This yields a fixed number of SGD iterations in each inner loop:
\begin{equation}
N_\text{iter} = \frac{N}{N_b},
\end{equation}
where $N$  is the training dataset size, and $N_b$ is the batch size.  We further introduce  a user-specified parameter $N_\text{aug}$ that controls the number of outer iterations  performed for each random permutation of the training dataset. This enables multiple optimization passes over identical datasets and batches before new random batches are generated, with each pass solving a distinct inner optimization problem characterized by updated  Lagrange multiplier $\lambda_k$ and penalty factor $\mu_k$.

\paragraph{Outer-loop updating of multipliers and penalties} 
We update the Lagrange multiplier $\lambda_k$ and penalty parameter $\mu_k$ in the outer loop based on the progress of inner-loop optimization. Traditional augmented Lagrangian methods employ dynamic constraint tolerances to guide these updates, presupposing that each inner optimization problem consistently achieves the required dynamic tolerance. However, since our inner-loop solves the optimization problem via SGD with a fixed number of iterations,  such convergence cannot be guaranteed.

We adopt a  criterion based on  sufficient reduction in PDE constraint violation, following the approach in \cite{dener2020training}. Specifically, the multiplier is updated when the current PDE residual shows sufficient improvement relative to the best value recorded thus far:
\begin{equation}\label{eq:criteria}
\|\widetilde{\mathcal{N}}(\theta_2^{k+1})\|_2 \leq \eta \|\widetilde{\mathcal{N}}_\text{best}\|_2,
\end{equation}
where $\eta\in(0,1)$ is a reduction factor  and $\widetilde{\mathcal{N}}_\text{best}$ tracks the smallest PDE residual achieved during training.
When this condition is satisfied, the multiplier follows the standard update:
\begin{equation}
\lambda_{k+1} = \lambda_k + \mu_k \widetilde{\mathcal{N}}(\theta_{2}^{k+1}).
\end{equation}
Otherwise, the penalty parameter is increased: 
\begin{equation}
\mu_{k+1} = \min(a \mu_{k}, \mu_{\max}),
\end{equation}
where $a>1$ is a scalar parameter, and $\mu_{\max}$ is an upper bound to prevent numerical instability.

We end this subsection by summarizing the complete stochastic augmented Lagrangian algorithm. 
First, we specify the number of random permutations of the training dataset, $N_{\text{shuffle}}$. 
Then, for each permutation, we perform $N_{\text{aug}}$ outer iterations; each outer iteration consists of:
\begin{itemize}
\item \textbf{Inner loop:} $N_\text{iter}$ SGD updates of the neural network parameters $\theta_2$.
\item \textbf{Outer update:} Update the Lagrange multipliers and penalty terms based on 
      the constraint satisfaction criterion~\eqref{eq:criteria}.
\end{itemize}

\subsection{Summary of the PINN-IMSM framework}
\label{sec:framework-summary}

Algorithm \ref{alg} outlines the complete PINN-IMSM procedure for reconstructing the velocity field $v(\bx)$ from the score-based formulation \eqref{eq:score-pde}. The method proceeds in two sequential stages, using only the trajectory measurements $\{\widetilde{X}_i\}_{i=1}^{N}$ as training data:
\begin{itemize}
    \item \textbf{Stage 1 (score estimation):} The score network $\bs_{\theta_1}(\bx)$ is trained on the trajectory data  by minimizing the multi-scale denoising score matching objective \eqref{eq:condition}. This yields the approximation $\bs_{\theta_1^*}(\bx, \sigma_L) \approx \bs(\bx) = \nabla_{\bx}\log\rho(\bx)$.
    
    \item \textbf{Stage 2 (velocity reconstruction):} With the pre-trained score estimate $\bs_{\theta_1^*}$ fixed, the velocity network $\bv_{\theta_2}(\bx)$ is trained using the same trajectory data.  This is achieved by solving the PDE-constrained optimization problem \eqref{eq:sub-prob} via the stochastic augmented Lagrangian framework. The solution to this problem inherently satisfies the physics constraint \eqref{eq:score-pde} and, by Theorem~\ref{them:unique}, corresponds to the unique minimum $L^2$-norm velocity field.
\end{itemize}

\paragraph{Remark on implementation} In Stage 2, the stochastic augmented Lagrangian method introduces several hyperparameters, such as $\eta$ and $a$. However, the training results remain relatively insensitive to their specific choices. This is due to the dynamic update of the Lagrange multiplier and penalty factor, governed by the convergence metric \eqref{eq:criteria}.  Moreover,  the augmented Lagrangian merit function \eqref{eq:lag} is easier to optimize compared to the direct PDE loss \eqref{eq:dp-loss}, thereby helping to mitigate common training difficulties encountered in standard PINN formulations. Numerical results in Section \ref{sec:num} show that this approach consistently achieves smaller PDE residuals than the standard PINN method.
\begin{algorithm}[!hbtp]
\caption{PINN-IMSM for velocity field reconstruction from unlabeled trajectory measurements. }\label{alg}
\begin{algorithmic}
\vspace{.05in}
\Require trajectory measurements $\{\widetilde{X}_i\}_{i=1}^{N}$, maximum number of epochs $K$ for the score network, update tolerance $\eta$, convergence tolerance $\epsilon$, initial penalty $\mu_\text{init}$, penalty update factor $a$, penalty upper bound  $\mu_{\max}$, batch size $N_b$, number of random shuffles $N_{\text{shuffle}}$ and
 number of iterations per shuffle $N_\text{aug}$.
\vspace{.05in}
\Ensure  {\emph{ {\bf{\emph{optimized}}} score network parameter set $ \theta_1^*$}}
\State choose noise scales $\{\sigma_i\}_{i=1}^L$; initialize $\theta_1^0$ from normal distribution around zero.
\For{$k =1 ,2,\dots, K$}
   \State compute the loss $\mathcal{L}_s(\theta_{1}^{k-1})$ via \eqref{eq:condition} 
   \State update $\theta_{1}^{k} \gets \underset{\substack{\theta_1}}{\text{argmin}} \;\mathcal{L}_s(\theta_{1}) $
\EndFor
\Ensure  {\emph{ {\bf{\emph{optimized}}} velocity network parameter set $ \theta_2^*$}}
\State $\lambda_0 \leftarrow 0$; initialize $\theta_{2}^{0}$ from a normal distribution around zero; set $ \widetilde{\mathcal{N}}_\text{best} \leftarrow  \widetilde{\mathcal{N}}(\theta_{2}^{0})$
\For{$j = 1, 2,\dots, N_{\text{shuffle}}$}
\State Randomly shuffle the training data and form batches of size $N_b$
\State $\mu_0 \gets \mu_\text{init}*(j+1)$
    \For{$k=1,2,\dots,N_\text{aug}$}
    \State solve \eqref{eq:lag} using SGD
    \State update $\theta_{2}^k \gets \underset{\substack{\theta_2}}{\text{argmin}} \left( ||\bv(\theta_2)||^2+\lambda_{k-1}^T  \widetilde{\mathcal{N}}(\theta_2)+\frac{\mu_{k-1}}{2}|| \widetilde{\mathcal{N}}(\theta_2)||^2\right)$
    \If{$|| \widetilde{\mathcal{N}}(\theta_{2}^{k})||_2\leq \eta || \widetilde{\mathcal{N}}_\text{best}||_2$}
        \If{$|| \widetilde{\mathcal{N}}(\theta_{2}^{k})||_2 \leq \epsilon$}
        \State $\bf{Solution}$ found; terminate with $\theta_2^* = \theta_{2}^{k}$
        \EndIf
        \State $\lambda_{k} \gets \lambda_{k-1} +\mu_{k-1}  \widetilde{\mathcal{N}}(\theta_{2}^{k}) $
        \State $ \widetilde{\mathcal{N}}_\text{best} \leftarrow  \widetilde{\mathcal{N}}(\theta_{2}^{k})$
        \State $\mu_{k} \leftarrow \mu_{k-1}$
    \Else
     \State $\lambda_{k} \leftarrow \lambda_{k-1}$
     \State $\mu_{k} \gets \min(a*\mu_{k-1},\mu_{\max})$
    \EndIf
    \EndFor
    \State $\lambda_0 \gets \lambda_{N_\text{aug}}$
\EndFor
\end{algorithmic}
\end{algorithm}

The logical flow of the overall PINN-IMSM pipeline is summarized in Figure \ref{fig:outline}.

\begin{figure}[!htbp]
    \centering
    \begin{tikzpicture}[scale=0.85, transform shape]
        \node [style={rect2}] (n1) at (0\columnA, 0\rowA) {SDE \eqref{eq:sde}  with constant isotropic diffusion};
        \node [style={rect3}] (n2) at (1.2\columnA, 1\rowA) {Steady-state FP equation \eqref{eq:steady-pde}};
        \node [style={rect4}] (n3) at (2.4\columnA, 0\rowA) { Score-based FP equation \eqref{eq:score-pde} };
        \node [style={rect5}] (n4) at (2.4\columnA, -2\rowA) {Score estimate };
        \node [style={rect6}] (n5) at (1.3\columnA, -3\rowA) {PDE-constrained optimization problem \eqref{eq:sub-prob} with learned score};
        \node [style={rect7}] (n6) at (0\columnA, -2.5\rowA) {Velocity reconstruction};
        
        \path [line3] (n1) -- node [midway, above, align=center ] {$\begin{array}{c} \mbox{invariant} \\ \mbox{ density} \end{array}$} (n2);
        \path [line3] (n2) -- node [midway, above, align=center ] { $\begin{array}{c} \mbox{score function}  \end{array}$ } (n3);
        \path [line3] (n3) -- node [midway, above, align=center ] { denoising score matching} (n4);
        \path [line3] (n4) -- node [midway, above, align=center ] {Well-posedness} node [midway, below, xshift=1.2cm, align=left ] { by Theorem \ref{them:unique}  } (n5);
        \path [line3] (n5) -- 
  node[pos=0.52, above, align=center] {stochastic augmented}
  node[pos=0.52, below, align=center] {Lagrangian approach}
(n6);
        \path [line3, dotted] (n6) -- node [midway, right, align=center ] { Algorithm \ref{alg} } (n1);
    \end{tikzpicture}
    \caption{Schematic outline of the PINN-IMSM workflow. The goal is to reconstruct the drift term of an SDE from unlabeled trajectory data by combining invariant measure score matching with PINNs. Starting from an SDE with constant isotropic diffusion, the framework transforms the problem into a steady-state FP equation, which is then reformulated using the score function of the invariant density. The method proceeds in two stages: first, estimating the score function via multi-scale denoising score matching; second, reconstructing the velocity field by solving a well-posed PDE-constrained optimization problem within the stochastic augmented Lagrangian framework.
    }
    \label{fig:outline}
\end{figure}
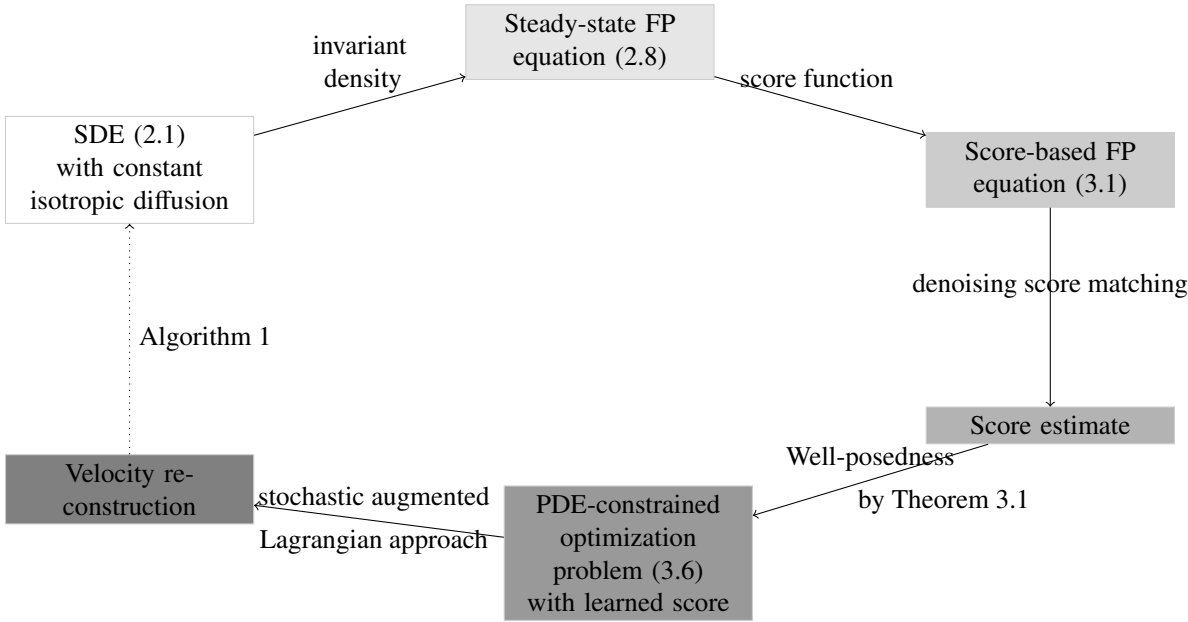

\section{Numerical results}\label{sec:num}
This section presents numerical examples to demonstrate the performance of the proposed PINN-IMSM framework on several dynamical systems. The goal is to evaluate how well it can recover the underlying velocity field from noisy, unlabeled trajectory data.

The method employs two neural networks. The score network $\bs_{\theta_1}$ is a six-layer multilayer perceptron (MLP) with 64 neurons per hidden layer, and the velocity network $\bv_{\theta_2}$ is a six-layer MLP with 128 neurons per hidden layer. Both networks employ the Swish activation function \cite{eger2018time} and are optimized using Adam \cite{Ilya_fix_2017}. These architectural choices, determined empirically,  yield stable performance across all experiments. 

 For each example, the training data are generated  by simulating the SDE \eqref{eq:sde_D}, where the true drift velocity is given in the ODE form $\dot{\bx}=\bv(\bx)$, and the diffusion coefficient $D$ controls the noise level.  We integrate the SDE numerically using the Euler–Maruyama method: given a time step $\Delta t$, the trajectory is updated as
\begin{align}
     X_{j+1}=X_j+\bv(X_j)\Delta t + \sqrt{2D}\xi_j\sqrt{\Delta t},
 \end{align}
 where $\{\xi_j\}$ are  independent standard normal random vectors in $\mathbb{R}^d$.  From the resulting long-time simulations, we collect point samples to form the training dataset.

 To visualize the results, we estimate the invariant density by simulating long-time trajectories (again with the Euler–Maruyama scheme) and binning the samples into a two-dimensional histogram. We report two density estimates: a \emph{reference} density obtained from simulations using the true drift velocity and a \emph{learned} density obtained from simulations using our reconstructed velocity field $\bv_{\theta_2^*}$. Most examples show the densities using hexagonal bins for a smooth visualization. For direct comparison with prior work \cite{botvinick2023learning}, Example \ref{ex:1} employs rectangular bins.
 In all figures, ``Aug\_lag'' denotes PINN-IMSM implemented with the stochastic augmented Lagrangian method in Section \ref{sec:m-pinn}. 

\begin{example}[\bf{Van der Pol oscillator}]
\label{ex:1}
In this example, we consider the Van der Pol oscillator \cite{guckenheimer2003dynamics} with the drift velocity
\begin{equation}\label{eq:van}
\left\{\begin{array}{l}
\dot{x}=y ,\\
\dot{y}=c\left(1-x^2\right) y-x, 
\end{array}\right.
\end{equation}
where $c=0.5$. The diffusion coefficient is set as $D=0.05$.
\end{example}

Figure \ref{fig:van} compares the reference system (top row) with 
the PINN-IMSM reconstruction (bottom row). The top row displays the ground-truth velocity field (left), the corresponding reference invariant density (middle), and  noisy trajectory samples that that serve as the training data (right). The bottom row shows the results of the proposed PINN-IMSM method: the reconstructed velocity field $\bv_{\theta_2^*}$ (left), the learned invariant density induced by this field (middle), and representative trajectories generated by Langevin dynamics driven by the learned score network $\bs_{\theta_1^*}$ (right). The Langevin samples provide an independent check of the score estimation quality. We observe that the reconstructed velocity deviates from the ground-truth field in regions far from the attracting limit cycle, where trajectory data are sparse and the drift velocity is therefore weakly
identified. However, these regions carry little probability mass under the invariant measure. Consequently, the learned dynamics still reproduce the invariant density with high fidelity in the high-probability region around the
attracting set. This indicates that the PINN-IMSM method captures the correct long-time statistical structure even where the drift velocity is not tightly constrained in
low-density, poorly sampled regions. Both densities use rectangular bins for direct comparison with \cite{botvinick2023learning}. Under this matched discretization, the learned invariant density agrees well with those in \cite{botvinick2023learning}.

 Figure~\ref{fig:vanloss} further compares the PDE residual during training for the standard PINN and the proposed PINN-IMSM (``Aug\_lag''). It can be observed that the proposed PINN-IMSM method achieves a much smaller residual, demonstrating the effectiveness of the stochastic augmented Lagrangian method in enforcing the physics constraint.

\begin{figure}[!htbp]
	\centering
		\includegraphics[width=0.32\linewidth]{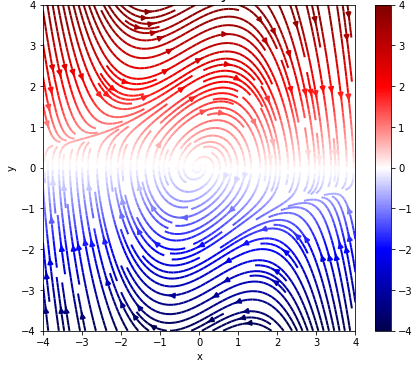}	\includegraphics[width=0.29\linewidth]{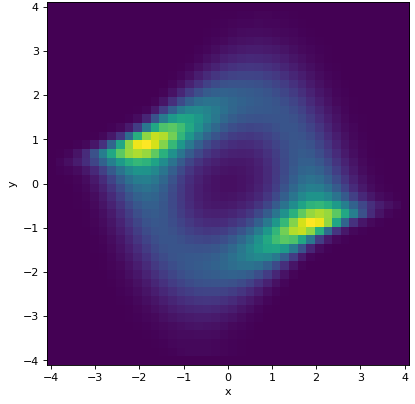}
		\includegraphics[width=0.3\linewidth]{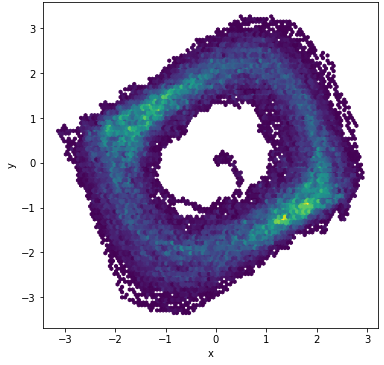}\\
		\includegraphics[width=0.32\linewidth]{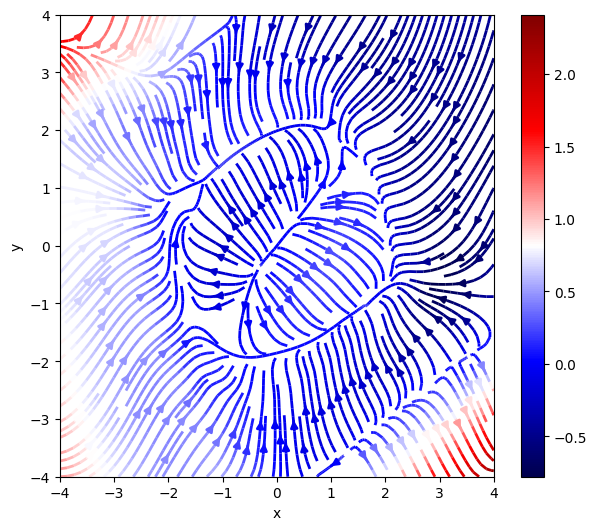}
		\includegraphics[width=0.29\linewidth]{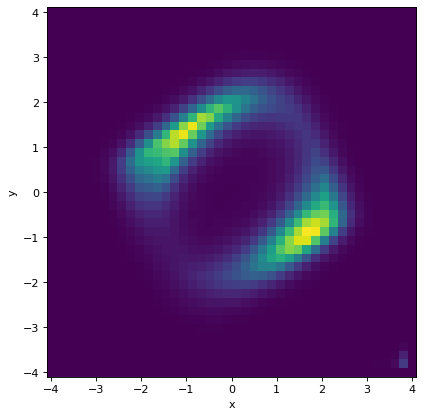}
		\includegraphics[width=0.3\linewidth]{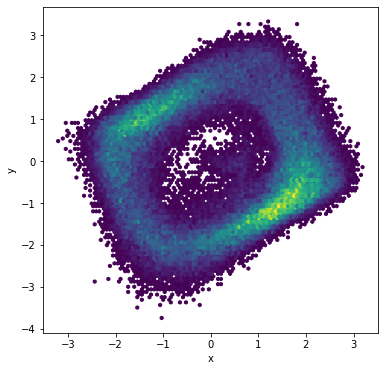}\\
  \caption{Van der Pol oscillator \eqref{eq:van} with $D=0.05$.
  \textbf{Top:} ground-truth velocity field (left), reference invariant density (middle), and noisy trajectory samples (right).
  \textbf{Bottom:} PINN-IMSM reconstructed velocity field $\bv_{\theta_2^*}$ (left), learned invariant density (middle), and samples generated by Langevin dynamics from the learned score $\bs_{\theta_1^*}$ (right).}
  \label{fig:van} 
\end{figure}

\begin{figure}[!htbp]
    \centering
    \includegraphics[width=0.6\linewidth]{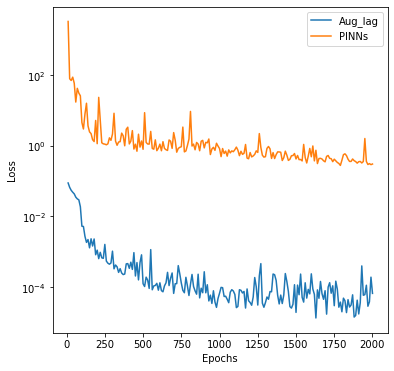}	
    \caption{PDE residual during training for the standard PINN and PINN-IMSM (``Aug\_lag'') on the Van der Pol oscillator.}
    \label{fig:vanloss}
\end{figure}

\begin{example}[\bf{Active swimmer in an anharmonic trap}]
We consider a model from active matter describing the motion of a motile swimmer in an anharmonic trap \cite{tailleur2008statistical}. The two-dimensional state $(x,\nu)$ evolves with the drift velocity
\begin{equation}\label{eq:swim}
   \begin{aligned}
       \dot{x}& = -x^3+\nu, \\
       \dot{\nu} & = -\gamma \nu ,
   \end{aligned}
\end{equation}
where $\gamma=0.1$, and the diffusion coefficient is set to $D=1.0$.
\end{example}

Figure~\ref{fig:swim} compares the learned invariant density from the PINN-IMSM method with the reference density for the active swimmer system. For conciseness, we show only the density visualizations in this example.  The close visual agreement demonstrates that the PINN-IMSM method captures the invariant statistics induced by the stochastically forced dynamics.

\begin{figure}[!htbp]
    \centering
    \includegraphics[width=0.33\linewidth]{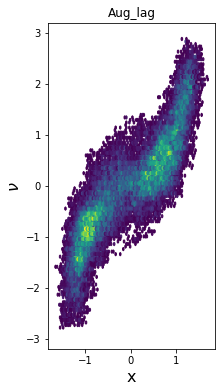}	\qquad \qquad\includegraphics[width=0.33\linewidth]{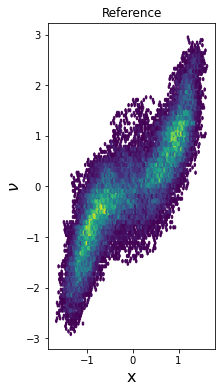}\\
    \caption{Invariant density estimates for the active swimmer model \eqref{eq:swim} with $\gamma=0.1$ and  $D=1.0$. \textbf{Left:} learned invariant density from the reconstructed velocity. \textbf{Right:} reference invariant density from the ground-truth velocity.}
    \label{fig:swim}
\end{figure}

\begin{example}[\bf{Lorenz-63 system}]
We consider the Lorenz-63 system \cite{luzzatto2005lorenz}, a canonical chaotic model. The drift velocity is given by
\begin{equation}\label{eq:lorenz}
\left\{\begin{array}{l}
\dot{x}=c_1(y-x), \\
\dot{y}=x(c_2-z)-y, \\
\dot{z}=xy-c_3 z,
\end{array}\right.
\end{equation}
where $(c_1,c_2,c_3)=(10,28,8/3)$. The diffusion coefficient is taken as $D=10$.
\end{example}

The Lorenz‑63 system presents a known challenge for the simultaneous reconstruction of all three velocity components. As noted in prior work \cite{botvinick2023learning}, a learned velocity field may  approximate the stationary state (i.e., satisfy the steady-state FP equation) yet fail to satisfy other physical properties of the invariant measure.  Whether this difficulty arises from inherent non‑uniqueness in the inverse problem or simply inconvenient local minima during training remains an open question for future investigation.

Given this challenge, we evaluate  the PINN-IMSM method on partial reconstructions.  Figure~\ref{fig:lorenz} presents the results. The left column shows the ground-truth drift velocity field and the corresponding reference invariant density. The middle column  shows a partial reconstruction where only $\dot{x}$ is learned, with $\dot{y}$ and $\dot{z}$ treated as known. The right column further learns both $\dot{x}$ and $\dot{y}$, again with $\dot{z}$ treated as known. In all cases, the invariant density is visualized via the empirical occupation measure projected on the  $(x,z)$ plane. Following \cite{yang2023optimal}, the occupation measure is formed by randomly subsample trajectories, which confirms that the PINN-IMSM method remains applicable when measurement times are irregular or unavailable.



\begin{figure}[!htbp]
	\centering
		\includegraphics[width=0.32\linewidth]{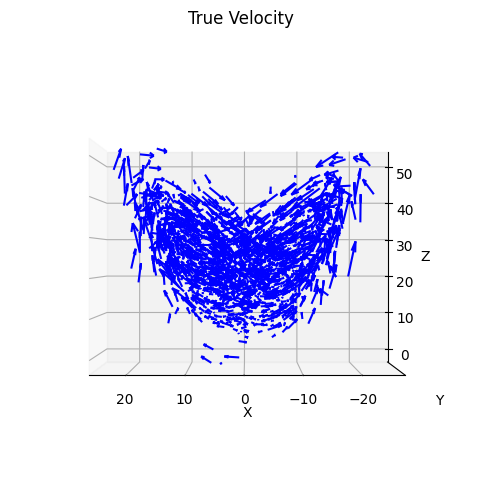}	\includegraphics[width=0.32\linewidth]{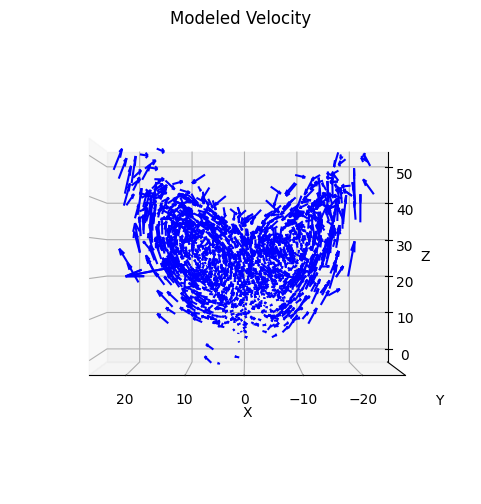}
		\includegraphics[width=0.32\linewidth]{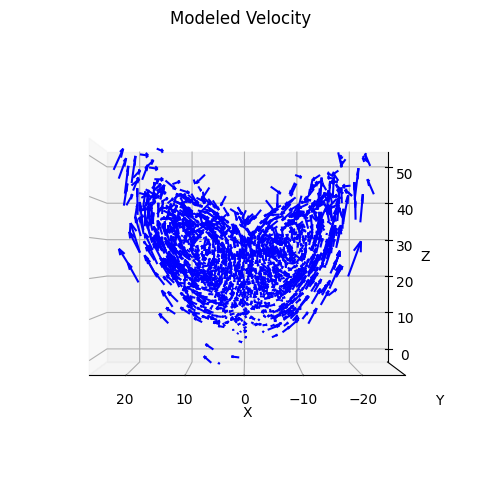}\\
		\includegraphics[width=0.3\linewidth]{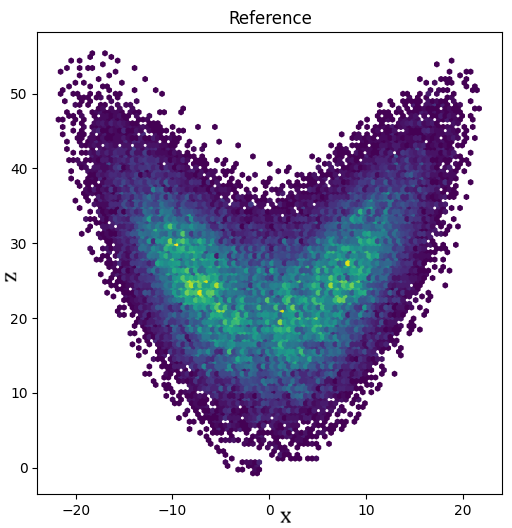}
		\includegraphics[width=0.3\linewidth]{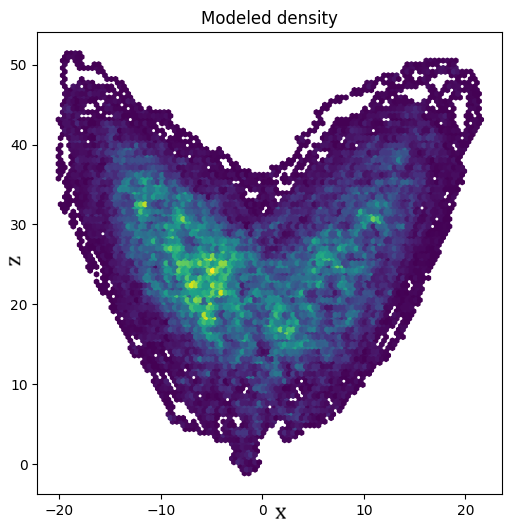}
		\includegraphics[width=0.3\linewidth]{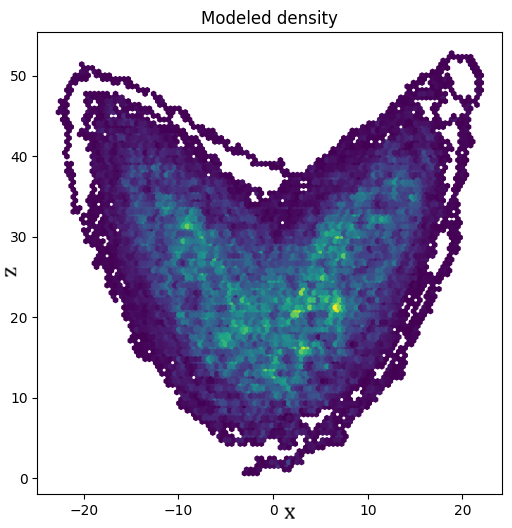}\\
  \caption{Lorenz-63 system \eqref{eq:lorenz} with $D=10$.
  \textbf{Top:} ground-truth velocity field (left), reconstructed velocity learning only $\dot{x}$ (middle), and reconstructed velocity  learning $\dot{x}$ and $\dot{y}$ (right).
  \textbf{Bottom:} corresponding invariant densities on the $(x,z)$ plane.}
  \label{fig:lorenz}
\end{figure}

\begin{example}[\bf{Lorenz-96 system}]
We examine the Lorenz-96 system \cite{ott2004local,karimi2010extensive}, defined for $i=1,\ldots,N$ by
\begin{equation}\label{eq:lorenz96}
    \frac{d x_i}{dt} = (x_{i+1}-x_{i-2})x_{i-1}-x_i+F,
\end{equation}
with periodic indexing: $x_{-1}=x_{N-1}$, $x_0=x_N$, and $x_{N+1}=x_1$. Here $F$ is a constant forcing term, the term $-x_i$ represents linear damping, and the quadratic term models advection while conserving kinetic energy in the absence of damping. We set $N=5$ and $F=8$. The diffusion coefficient is $D=0.05$.
\end{example}

Given the five-dimensional state space, we visualize the  invariant density induced by the reconstructed drift velocity through two-dimensional projection histograms. Figure~\ref{fig:lorenz96} shows representative projection densities, capturing the characteristic statistical structure of the chaotic attractor. These results demonstrate that the PINN-IMSM method  is effective in higher‑dimensional settings where mesh-based density solvers become computationally prohibitive.

\begin{figure}[!htbp]
    \centering
    \includegraphics[width=0.23\linewidth]{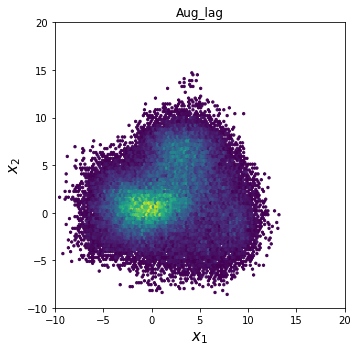}
    \hfill
    \includegraphics[width=0.23\linewidth]{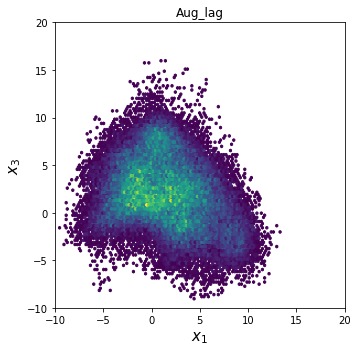}
    \hfill
    \includegraphics[width=0.23\linewidth]{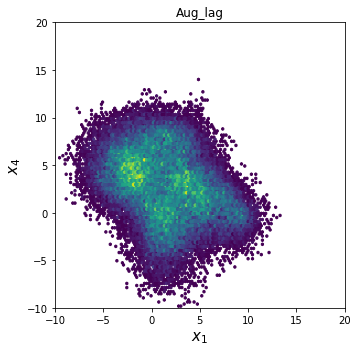}
    \hfill
    \includegraphics[width=0.23\linewidth]{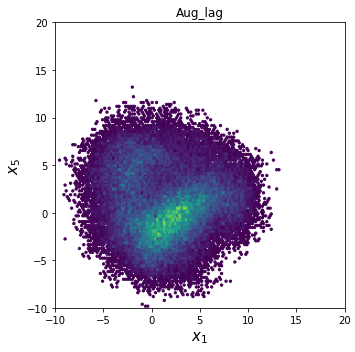}
    
    \vspace{0.2cm}
    
    \includegraphics[width=0.23\linewidth]{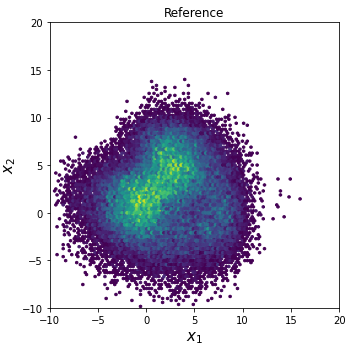}
    \hfill
    \includegraphics[width=0.23\linewidth]{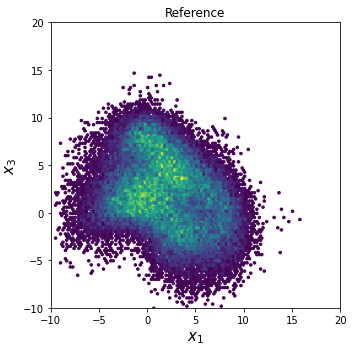}
    \hfill
    \includegraphics[width=0.23\linewidth]{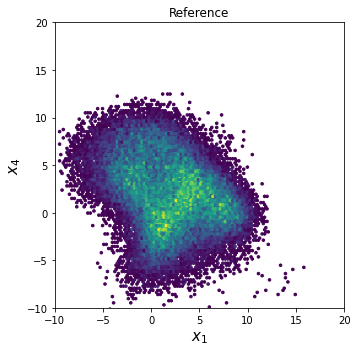}
    \hfill
    \includegraphics[width=0.23\linewidth]{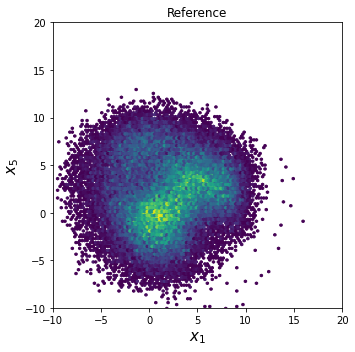}
    
    \caption{Lorenz-96 system \eqref{eq:lorenz96} with $N=5$, $F=8$, and  $D=0.05$. \textbf{Top:} invariant density projections from the PINN‑IMSM reconstruction.  \textbf{Bottom:}  corresponding reference invariant density from the true system. Two‑dimensional projection histograms onto $(x_{1},x_{2})$, $(x_{1},x_{3})$, $(x_{1},x_{4})$, and $(x_{1},x_{5})$ planes. }
    \label{fig:lorenz96}
\end{figure}

\section{Conclusion}\label{sec:con}
In this paper, we introduced PINN-IMSM, a novel mesh-free framework for reconstructing dynamical systems from unlabeled point-cloud observational data that capture the asymptotic statistical behavior of the underlying system without explicit temporal information. By adopting an Eulerian perspective and reformulating the stationary FP equation in terms of the score function, our approach overcomes the reliance on time‑labeled trajectories required by Lagrangian approaches. The core innovation of the proposed PINN-IMSM method is the integration of denoising score matching with physics-informed neural networks: the score function of the invariant measure is estimated directly from data, circumventing costly density estimation, and is then embedded into a PINN to reconstruct the velocity field under the PDE constraint. To ensure well-posedness in high-dimensional settings, the reconstruction is reformulated as a PDE‑constrained optimization that seeks the velocity field with minimal $L^2$-norm, solved efficiently using the stochastic augmented Lagrangian method. Numerical experiments on benchmark problems demonstrate that PINN-IMSM achieves accurate and stable reconstruction of dynamical systems, successfully handling problems up to five dimensions. This represents a significant advancement over existing mesh-based Eulerian approaches, which are typically limited to low-dimensional problems $(d \leq 3)$ due to the curse of dimensionality. Moreover, the method shows robustness in scenarios with incomplete velocity information, faithfully recovering chaotic attractors and complex dynamical behaviors from sparse observational data.  The theoretical foundation of our approach is strengthened by our proof that, under suitable conditions, the inverse problem admits a unique solution that depends continuously on the score function. 

Future work will extend the current framework to systems with non‑constant or anisotropic diffusion coefficients and develop theoretical tools for more general classes of stochastic dynamical systems. Further promising directions include incorporating uncertainty quantification and exploring applications to high‑dimensional problems in molecular dynamics, climate modeling, and biological systems.

\section*{Acknowledgments}
Research work of Y. Chen and X. Zhong  was partially supported by National Key R\&D Program of China 2024YFA1012302 and  NSFC Grant 11871428.

\bibliographystyle{unsrt} %
\bibliography{main}

@misc{goodfellow2016deep,
  title={Deep learning},
  author={Goodfellow, Ian},
  year={2016},
  publisher={MIT press}
}

@article{raissi2020hidden,
  title={Hidden fluid mechanics: Learning velocity and pressure fields from flow visualizations},
  author={Raissi, Maziar and Yazdani, Alireza and Karniadakis, George Em},
  journal={Science},
  volume={367},
  number={6481},
  pages={1026--1030},
  year={2020},
  publisher={American Association for the Advancement of Science}
}

@book{evans2022partial,
  title={Partial differential equations},
  author={Evans, Lawrence C},
  volume={19},
  year={2022},
  publisher={American Mathematical Society}
}

@article{chen1995universal,
  title={Universal approximation to nonlinear operators by neural networks with arbitrary activation functions and its application to dynamical systems},
  author={Chen, Tianping and Chen, Hong},
  journal={IEEE transactions on neural networks},
  volume={6},
  number={4},
  pages={911--917},
  year={1995},
  publisher={IEEE}
}

@article{botvinick2023learning,
    author = {Botvinick-Greenhouse, Jonah and Martin, Robert and Yang, Yunan},
    title = {Learning dynamics on invariant measures using PDE-constrained optimization},
    journal = {Chaos: An Interdisciplinary Journal of Nonlinear Science},
    volume = {33},
    number = {6},
    year = {2023}
}

@article{hestenes1969multiplier,
  title={Multiplier and gradient methods},
  author={Hestenes, Magnus R},
  journal={Journal of optimization theory and applications},
  volume={4},
  number={5},
  pages={303--320},
  year={1969},
  publisher={Springer}
}

@article{zang2020weak,
  title={Weak adversarial networks for high-dimensional partial differential equations},
  author={Zang, Yaohua and Bao, Gang and Ye, Xiaojing and Zhou, Haomin},
  journal={Journal of Computational Physics},
  volume={411},
  pages={109409},
  year={2020},
  publisher={Elsevier}
}

@article{hornik1989multilayer,
  title={Multilayer feedforward networks are universal approximators},
  author={Hornik, Kurt and Stinchcombe, Maxwell and White, Halbert},
  journal={Neural networks},
  volume={2},
  number={5},
  pages={359--366},
  year={1989},
  publisher={Elsevier}
}

@book{krantz2002implicit,
  title={The implicit function theorem: history, theory, and applications},
  author={Krantz, Steven George and Parks, Harold R},
  year={2002},
  publisher={Springer Science \& Business Media}
}

@book{quarteroni2008numerical,
  title={Numerical approximation of partial differential equations},
  author={Quarteroni, Alfio and Valli, Alberto},
  volume={23},
  year={2008},
  publisher={Springer Science \& Business Media}
}

@book{brenner2008mathematical,
  title={The mathematical theory of finite element methods},
  author={Brenner, Susanne C},
  year={2008},
  publisher={Springer}
}

@article{baydin2017automatic,
  title={Automatic differentiation in machine learning: a survey},
  author={Baydin, Atilim Gunes and Pearlmutter, Barak A and Radul, Alexey Andreyevich and Siskind, Jeffrey Mark},
  journal={Journal of Machine Learning Research},
  volume={18},
  number={153},
  pages={1--43},
  year={2017}
}

@article{kingma2014adam,
  title={Adam: a method for stochastic optimization},
  author={Kingma, Diederik P and Ba, Jimmy},
  journal={arXiv preprint arXiv:1412.6980},
  year={2014}
}

@article{liu1989limited,
  title={On the limited memory BFGS method for large scale optimization},
  author={Liu, Dong C and Nocedal, Jorge},
  journal={Mathematical Programming},
  volume={45},
  number={1-3},
  pages={503--528},
  year={1989},
  publisher={Springer}
}

@book{gilbarg1977elliptic,
  title={Elliptic partial differential equations of second order},
  author={Gilbarg, David and Trudinger, Neil S and Gilbarg, David and Trudinger, NS},
  volume={224},
  number={2},
  year={1977},
  publisher={Springer}
}

@book{ern2021finite,
  title={Finite elements II},
  author={Ern, Alexandre and Guermond, Jean-Luc and others},
  year={2021},
  publisher={Springer}
}

@book{bellman1961adaptive,
  title={Adaptive control processes: a guided tour},
  author={Bellman, Richard Ernest},
  year={1961},
  publisher={Princeton University Press},
  address={Princeton, NJ}
}

@inproceedings{song2020score,
  title={Score-based generative modeling through stochastic differential equations},
  author={Song, Jiaming and Meng, Chenlin and Ermon, Stefano},
  booktitle={International Conference on Learning Representations},
  year={2021}
}

@inproceedings{ho2020denoising,
  title={Denoising diffusion probabilistic models},
  author={Ho, Jonathan and Jain, Ajay and Abbeel, Pieter},
  booktitle={Advances in Neural Information Processing Systems},
  volume={33},
  pages={6840--6851},
  year={2020}
}

@article{song2020improved,
  title={Improved techniques for training score-based generative models},
  author={Song, Yang and Ermon, Stefano},
  journal={Advances in neural information processing systems},
  volume={33},
  pages={12438--12448},
  year={2020}
}

@inproceedings{song2020sliced,
  title={Sliced score matching: a scalable approach to density and score estimation},
  author={Song, Yang and Garg, Sahaj and Shi, Jiaxin and Ermon, Stefano},
  booktitle={Uncertainty in Artificial Intelligence},
  pages={574--584},
  year={2020},
  organization={PMLR}
}

@article{song2019generative,
  title={Generative modeling by estimating gradients of the data distribution},
  author={Song, Yang and Ermon, Stefano},
  journal={Advances in neural information processing systems},
  volume={32},
  year={2019}
}

@article{hyvarinen2005estimation,
  title={Estimation of non-normalized statistical models by score matching.},
  author={Hyv{\"a}rinen, Aapo and Dayan, Peter},
  journal={Journal of Machine Learning Research},
  volume={6},
  number={4},
  year={2005}
}

@article{vincent2011connection,
  title={A connection between score matching and denoising autoencoders},
  author={Vincent, Pascal},
  journal={Neural computation},
  volume={23},
  number={7},
  pages={1661--1674},
  year={2011},
  publisher={MIT Press}
}

@inproceedings{kadeethum2020physics,
  title={Physics-informed neural networks for solving inverse problems of nonlinear Biot's equations: batch training},
  author={Kadeethum, Teeratorn and J{\o}rgensen, Thomas M and Nick, Hamidreza M},
  booktitle={ARMA US Rock Mechanics/Geomechanics Symposium},
  pages={ARMA--2020},
  year={2020},
  organization={ARMA}
}

@article{depina2022application,
  title={Application of physics-informed neural networks to inverse problems in unsaturated groundwater flow},
  author={Depina, Ivan and Jain, Saket and Mar Valsson, Sigurdur and Gotovac, Hrvoje},
  journal={Georisk: Assessment and Management of Risk for Engineered Systems and Geohazards},
  volume={16},
  number={1},
  pages={21--36},
  year={2022},
  publisher={Taylor \& Francis}
}

@article{guo2023high,
  title={High-dimensional inverse modeling of hydraulic tomography by physics informed neural network (HT-PINN)},
  author={Guo, Quan and Zhao, Yue and Lu, Chunhui and Luo, Jian},
  journal={Journal of Hydrology},
  volume={616},
  pages={128828},
  year={2023},
  publisher={Elsevier}
}

@article{raissi2019physics,
  title={Physics-informed neural networks: A deep learning framework for solving forward and inverse problems involving nonlinear partial differential equations},
  author={Raissi, Maziar and Perdikaris, Paris and Karniadakis, George E},
  journal={Journal of Computational physics},
  volume={378},
  pages={686--707},
  year={2019},
  publisher={Elsevier}
}

@article{karniadakis2021physics,
  title={Physics-informed machine learning},
  author={Karniadakis, George Em and Kevrekidis, Ioannis G and Lu, Lu and Perdikaris, Paris and Wang, Sifan and Yang, Liu},
  journal={Nature Reviews Physics},
  volume={3},
  number={6},
  pages={422--440},
  year={2021},
  publisher={Nature Publishing Group UK London}
}

@article{wang2022and,
  title={When and why PINNs fail to train: a neural tangent kernel perspective},
  author={Wang, Sifan and Yu, Xinling and Perdikaris, Paris},
  journal={Journal of Computational Physics},
  volume={449},
  pages={110768},
  year={2022},
  publisher={Elsevier}
}

@article{jagtap2020conservative,
  title={Conservative physics-informed neural networks on discrete domains for conservation laws: applications to forward and inverse problems},
  author={Jagtap, Ameya D and Kharazmi, Ehsan and Karniadakis, George Em},
  journal={Computer Methods in Applied Mechanics and Engineering},
  volume={365},
  pages={113028},
  year={2020},
  publisher={Elsevier}
}

@inproceedings{stock2023bayesian,
  title={Bayesian physics-informed neural networks for robust system identification of power systems},
  author={Stock, Simon and Stiasny, Jochen and Babazadeh, Davood and Becker, Christian and Chatzivasileiadis, Spyros},
  booktitle={2023 IEEE Belgrade PowerTech},
  pages={1--6},
  year={2023},
  organization={IEEE}
}

@article{lu2021learning,
  title={Learning nonlinear operators via DeepONet based on the universal approximation theorem of operators},
  author={Lu, Lu and Jin, Pengzhan and Pang, Guofei and Zhang, Zhongqiang and Karniadakis, George Em},
  journal={Nature machine intelligence},
  volume={3},
  number={3},
  pages={218--229},
  year={2021},
  publisher={Nature Publishing Group UK London}
}

@inproceedings{
li2021fourier,
title={Fourier Neural Operator for Parametric Partial Differential Equations},
author={Zongyi Li and Nikola Borislavov Kovachki and Kamyar Azizzadenesheli and Burigede liu and Kaushik Bhattacharya and Andrew Stuart and Anima Anandkumar},
booktitle={International Conference on Learning Representations},
year={2021},
}

@inproceedings{anandkumar2020neural,
  title={Neural operator: graph kernel network for partial differential equations},
  author={Anandkumar, Anima and Azizzadenesheli, Kamyar and Bhattacharya, Kaushik and Kovachki, Nikola and Li, Zongyi and Liu, Burigede and Stuart, Andrew},
  booktitle={ICLR 2020 workshop on integration of deep neural models and differential equations},
  year={2020}
}

@article{tanyu2023deep,
  title={Deep learning methods for partial differential equations and related parameter identification problems},
  author={Tanyu, Derick Nganyu and Ning, Jianfeng and Freudenberg, Tom and Heilenk{\"o}tter, Nick and Rademacher, Andreas and Iben, Uwe and Maass, Peter},
  journal={Inverse Problems},
  volume={39},
  number={10},
  pages={103001},
  year={2023},
  publisher={IOP Publishing}
}

@article{de2022cost,
  title   = {The Cost-Accuracy Trade-Off in Operator Learning with Neural Networks},
  author  = {de Hoop, Maarten V. and Huang, Daniel Zhengyu and Qian, Elizabeth and Stuart, Andrew M.},
  journal = {Journal of Machine Learning},
  year    = {2022},
  volume  = {1},
  number  = {3}
}

@article{li2024physics,
  title={Physics-informed neural operator for learning partial differential equations},
  author={Li, Zongyi and Zheng, Hongkai and Kovachki, Nikola and Jin, David and Chen, Haoxuan and Liu, Burigede and Azizzadenesheli, Kamyar and Anandkumar, Anima},
  journal={ACM/JMS Journal of Data Science},
  volume={1},
  number={3},
  pages={1--27},
  year={2024},
  publisher={ACM New York, NY}
}

@article{rosofsky2023applications,
  title={Applications of physics informed neural operators},
  author={Rosofsky, Shawn G and Al Majed, Hani and Huerta, EA},
  journal={Machine Learning: Science and Technology},
  volume={4},
  number={2},
  pages={025022},
  year={2023},
  publisher={IOP Publishing}
}

@article{dener2020training,
  title={Training neural networks under physical constraints using a stochastic augmented Lagrangian approach},
  author={Dener, Alp and Miller, Marco Andres and Churchill, Randy Michael and Munson, Todd and Chang, Choong-Seock},
  journal={arXiv preprint arXiv:2009.07330},
  year={2020}
}

@article{huang2015steady,
  title={Steady states of Fokker--Planck equations: I. existence},
  author={Huang, Wen and Ji, Min and Liu, Zhenxin and Yi, Yingfei},
  journal={Journal of Dynamics and Differential Equations},
  volume={27},
  pages={721--742},
  year={2015},
  publisher={Springer}
}

@article{powell1969method,
  title={A method for nonlinear constraints in minimization problems},
  author={Powell, Michael JD},
  journal={Optimization},
  pages={283--298},
  year={1969},
  publisher={Academic Press}
}

@article{ott2004local,
  title={A local ensemble Kalman filter for atmospheric data assimilation},
  author={Ott, Edward and Hunt, Brian R and Szunyogh, Istvan and Zimin, Aleksey V and Kostelich, Eric J and Corazza, Matteo and Kalnay, Eugenia and Patil, DJ and Yorke, James A},
  journal={Tellus A: Dynamic Meteorology and Oceanography},
  volume={56},
  number={5},
  pages={415--428},
  year={2004},
  publisher={Taylor \& Francis}
}

@inproceedings{eger2018time,
  title={Is it Time to Swish? Comparing Deep Learning Activation Functions Across NLP tasks},
  author={Eger, Steffen and Youssef, Paul and Gurevych, Iryna},
  booktitle={Proceedings of the 2018 Conference on Empirical Methods in Natural Language Processing},
  pages={4415--4424},
  year={2018}
}

@article{yang2023optimal,
  title={Optimal transport for parameter identification of chaotic dynamics via invariant measures},
  author={Yang, Yunan and Nurbekyan, Levon and Negrini, Elisa and Martin, Robert and Pasha, Mirjeta},
  journal={SIAM Journal on Applied Dynamical Systems},
  volume={22},
  number={1},
  pages={269--310},
  year={2023},
  publisher={SIAM}
}

@article{luzzatto2005lorenz,
  title={The Lorenz attractor is mixing},
  author={Luzzatto, Stefano and Melbourne, Ian and Paccaut, Frederic},
  journal={Communications in Mathematical Physics},
  volume={260},
  pages={393--401},
  year={2005},
  publisher={Springer}
}

@article{tailleur2008statistical,
  title={Statistical mechanics of interacting run-and-tumble bacteria},
  author={Tailleur, Julien and Cates, Michael E},
  journal={Physical review letters},
  volume={100},
  number={21},
  pages={218103},
  year={2008},
  publisher={APS}
}

@article{guckenheimer2003dynamics,
  title={Dynamics of the van der Pol equation},
  author={Guckenheimer, John},
  journal={IEEE Transactions on Circuits and Systems},
  volume={27},
  number={11},
  pages={983--989},
  year={2003},
  publisher={IEEE}
}

@article{karimi2010extensive,
  title={Extensive chaos in the Lorenz-96 model},
  author={Karimi, Alireza and Paul, Mark R},
  journal={Chaos: An interdisciplinary journal of nonlinear science},
  volume={20},
  number={4},
  year={2010},
  publisher={AIP Publishing}
}

@inproceedings{Ilya_fix_2017,
  title={Decoupled weight decay regularization},
  author={Loshchilov, Ilya and Hutter, Frank},
  booktitle={International Conference on Learning Representations},
  year={2018}
}

@article{lu_physics-informed_2021,
author = {Lu, Lu and Pestourie, Rapha\"{e}l and Yao, Wenjie and Wang, Zhicheng and Verdugo, Francesc and Johnson, Steven G.},
title = {Physics-informed neural networks with hard constraints for inverse design},
journal = {SIAM Journal on Scientific Computing},
volume = {43},
number = {6},
pages = {B1105-B1132},
year = {2021},
doi = {10.1137/21M1397908},
}

@article{Ovchinnikov2016supersymm,
  title={Introduction to supersymmetric theory of stochastics},
  author={Ovchinnikov, Igor V},
  journal={Entropy},
  volume={18},
  number={4},
  pages={108},
  year={2016},
  publisher={MDPI}
}

@book{baxendale2007stochastic,
  title={Stochastic differential equations: theory and applications},
  author={Baxendale, Peter H and Lototsky, Sergey V},
  volume={2},
  year={2007},
  publisher={World Scientific}
}

@article{chen2018neural,
  title={Neural ordinary differential equations},
  author={Chen, Ricky TQ and Rubanova, Yulia and Bettencourt, Jesse and Duvenaud, David K},
  journal={Advances in neural information processing systems},
  volume={31},
  year={2018}
}

@article{linot2023stabilized,
  title={Stabilized neural ordinary differential equations for long-time forecasting of dynamical systems},
  author={Linot, Alec J and Burby, Joshua W and Tang, Qi and Balaprakash, Prasanna and Graham, Michael D and Maulik, Romit},
  journal={Journal of Computational Physics},
  volume={474},
  pages={111838},
  year={2023},
  publisher={Elsevier}
}

@article{portwood2019turbulence,
  title={Turbulence forecasting via neural ode},
  author={Portwood, Gavin D and Mitra, Peetak P and Ribeiro, Mateus Dias and Nguyen, Tan Minh and Nadiga, Balasubramanya T and Saenz, Juan A and Chertkov, Michael and Garg, Animesh and Anandkumar, Anima and Dengel, Andreas and others},
  journal={arXiv preprint arXiv:1911.05180},
  year={2019}
}

@article{linot2022data,
  title={Data-driven reduced-order modeling of spatiotemporal chaos with neural ordinary differential equations},
  author={Linot, Alec J and Graham, Michael D},
  journal={Chaos: An Interdisciplinary Journal of Nonlinear Science},
  volume={32},
  number={7},
  year={2022},
  publisher={AIP Publishing}
}

@article{brunton2016discovering,
  title={Discovering governing equations from data by sparse identification of nonlinear dynamical systems},
  author={Brunton, Steven L and Proctor, Joshua L and Kutz, J Nathan},
  journal={Proceedings of the national academy of sciences},
  volume={113},
  number={15},
  pages={3932--3937},
  year={2016},
  publisher={National Acad Sciences}
}

@article{fasel2022ensemble,
  title={Ensemble-SINDy: Robust sparse model discovery in the low-data, high-noise limit, with active learning and control},
  author={Fasel, Urban and Kutz, J Nathan and Brunton, Bingni W and Brunton, Steven L},
  journal={Proceedings of the Royal Society A},
  volume={478},
  number={2260},
  pages={20210904},
  year={2022},
  publisher={The Royal Society}
}

@article{messenger2021weak,
  title={Weak SINDy: Galerkin-based data-driven model selection},
  author={Messenger, Daniel A and Bortz, David M},
  journal={Multiscale Modeling \& Simulation},
  volume={19},
  number={3},
  pages={1474--1497},
  year={2021},
  publisher={SIAM}
}

@article{zeng2023reconstruction,
  title={Reconstruction of dynamical systems from data without time labels},
  author={Zeng, Zhijun and Hu, Pipi and Bao, Chenglong and Zhu, Yi and Shi, Zuoqiang},
  journal={arXiv preprint arXiv:2312.04038},
  year={2023}
}

@article{StochStable,
  title={Stochastic stability in some chaotic dynamical systems},
  author={Keller, Gerhard},
  journal={Monatshefte f{\"u}r Mathematik},
  volume={94},
  number={4},
  pages={313--333},
  year={1982},
  publisher={Springer}
}

@article{Young1986StochHyp,
  title={Stochastic stability of hyperbolic attractors},
  author={Young, Lai-Sang},
  journal={Ergodic Theory and Dynamical Systems},
  volume={6},
  number={2},
  pages={311--319},
  year={1986},
  publisher={Cambridge University Press}
}

@article{CowiesonYoung2005,
  title={SRB measures as zero-noise limits},
  author={Cowieson, William and Young, Lai-Sang},
  journal={Ergodic Theory and dynamical systems},
  volume={25},
  number={4},
  pages={1115--1138},
  year={2005},
  publisher={Cambridge University Press}
}

\appendix
\section{Proof of Theorem \ref{them:unique}} \label{sec:appendix}
In this section, we show the detailed proof of Theorem \ref{them:unique} in Section \ref{sec:m-constrained}. This section is organized as follows: in Section \ref{sub:pre}, we introduce the necessary notations, function spaces, norms, and preliminary lemmas. In Section \ref{sec:well_forword}, we present the existence and regularity of the forward problem, which establishes an important relationship between the score function and the velocity field.
Finally, in Section \ref{sec:finalproof}, we complete the proof of Theorem \ref{them:unique} that ensuring the well-posedness of the inverse problem.

\subsection{Preliminaries}\label{sub:pre}
We first introduce function spaces used throughout the proof.  We consider the standard continuous function space $C^m(\Omega)$ for $m \geq 2$ equipped with the norm $\|\cdot\|_{C^m}$, and the space $C^{m,1}(\Omega)$ of $m$-times continuously differentiable functions whose $m$-th derivatives are Lipschitz continuous. We also employ the  Sobolev spaces $W^{m,p}(\Omega)$ for $m \ge 0$ and $1 \le p \le \infty$ equipped with the norm $\|\cdot\|_{W^{m,p}}$. The Hilbert space is denoted by $H^m(\Omega)$ for $m \ge 0$ equipped with the norm $\|\cdot\|_{H^m}$; in particular, $H^0(\Omega)=L^2(\Omega)$ with the $L^2$ norm denoted simply by $\|\cdot\|$. The dual space of $H^1(\Omega)$ is denoted by  $H^{-1}(\Omega)$.
For vector fields we use the standard Hilbert space
\[
H(\mathrm{div};\Omega)
:= \Big\{\bv\in L^2(\Omega;\mathbb{R}^d):\ \nabla\!\cdot\bv\in L^2(\Omega)\Big\},
\qquad
\|\bv\|_{H(\mathrm{div})}^2
:= \|\bv\|_{L^2}^2+\|\nabla\!\cdot\bv\|_{L^2}^2.
\]
We also employ the following subspaces and local spaces:
\[
H^1_\diamond(\Omega):=\Big\{w\in H^1(\Omega):\int_\Omega w\,\rd x=0\Big\}.
\]
and  $H^m_{\mathrm{loc}}(\Omega)$ consisting of functions belonging to $H^m$ on every compact subset of $\Omega$.

The following classical results are essential for our analysis. 


\begin{theorem}[Implicit Function Theorem {\cite[Theorem 3.3.1]{krantz2002implicit}}]
\label{thm:implicitfunction}
Let $X,Y,Z$ be Banach spaces and let $U\subset X$, $V\subset Y$ be open sets. Suppose $F:U\times V\to Z$ is continuously Fr\'echet differentiable. If there exist $(x_0,y_0)\in U\times V$ such that
\[
F(x_0,y_0)=0
\quad\text{and}\quad
D_yF(x_0,y_0):Y\to Z \text{ is a bounded linear isomorphism},
\]
then there exist neighborhoods $U_0\subset U$ of $x_0$ and $V_0\subset V$ of $y_0$, and a unique continuously differentiable mapping $g:U_0\to V_0$ such that
\[
F(x,g(x))=0\quad \text{for all } x\in U_0,
\qquad g(x_0)=y_0.
\]
Moreover, if $F\in C^k$ then $g\in C^k$, and 
$Dg(x) = -\big(D_yF(x,g(x))\big)^{-1} D_xF(x,g(x))$ for $x\in U_0$.
\end{theorem}


\subsection{Existence and regularity of the forward problem} 
\label{sec:well_forword}
We consider the stationary FP equation \eqref{eq:steady-pde} with
the boundary condition \eqref{eq:bc:simple} in  divergence form  
\begin{align}\label{eq:ZFN-strong}
\nabla\!\cdot\big(-D\nabla\rho+\rho\,\bv\big)&=0 \quad \text{in }\Omega, 
\\
(D\nabla\rho-\rho\,\bv)\cdot \bn&=0 \quad \text{on }\partial\Omega, \label{eq:BC}
\end{align}
where $\bv$ is a given velocity. 

The weak formulation of the problem \eqref{eq:ZFN-strong} is: find  $\rho \in H^1(\Omega)$ such that 
\begin{align}
    \label{eq:a_bilinear}
    a(\rho,\varphi):=\int_\Omega \big(D\nabla\rho - \rho\,\bv\big)\cdot\nabla\varphi\,\rd x=0, \quad 
    \forall\varphi\in H^1(\Omega).
\end{align}
The bilinear form $a(\cdot,\cdot)$ is bounded on $H^1(\Omega)\times H^1(\Omega)$, since by Cauchy-Schwarz inequality,  for any $u,\varphi\in H^1(\Omega)$,
\begin{align}
   \label{eq:abound}
    |a(u,\varphi)|
\le D\|\nabla u\|_{L^2}\|\nabla\varphi\|_{L^2}
+\|\bv\|_{L^\infty}\,\|u\|_{L^2}\,\|\nabla\varphi\|_{L^2} \leq C \|u\|_{H^1}\|\varphi\|_{H^1}.
\end{align}
Moreover, a direct computation yields the G\r{a}rding inequality.
\begin{align}
     \label{eq:agarding}
     a(u,u)=D\|\nabla u\|_{L^2}^2-\int_\Omega u\,\bv\cdot\nabla u\,\rd x
\ \ge\ \frac{D}{2}\|\nabla u\|_{L^2}^2 - \frac{\|\bv\|_{L^\infty}^2}{2D}\,\|u\|_{L^2}^2.
\end{align}

The following proposition clarifies that  a weak solution satisfies the boundary condition. Its proof follows the arguments of \cite[Proposition 31.13]{ern2021finite} with the diffusion flux taken as $\bm \sigma=-D\nabla\rho+\rho\,\bv$.
\begin{proposition}[Weak solution]
\label{prop:weak}
Let $\rho$ solve \eqref{eq:a_bilinear} with $\bv\in L^\infty(\Omega;\mathbb{R}^d)$. Then, the PDE in \eqref{eq:ZFN-strong} holds almost everywhere  in $\Omega$ and the boundary condition holds almost everywhere on $\partial \Omega$.
\end{proposition}

\begin{lemma}[Existence, uniqueness and interior regularity]\label{lem:ZFN}
Assume $\bv\in L^\infty(\Omega;\mathbb{R}^d)$ and $\nabla\!\cdot\bv\in L^\infty(\Omega)$. Let $C_P$ be the Poincar\'e constant on $H^1_\diamond(\Omega)$, i.e.,
$\|w\|_{L^2}\le C_P\|\nabla w\|_{L^2}$ for all $w\in H^1_\diamond(\Omega)$.
If the diffusion coefficient satisfies
\( D>\|\bv\|_{L^\infty(\Omega)}\,C_P\),
then the boundary value problem \eqref{eq:ZFN-strong} admits a unique weak solution
$\rho\in H^1(\Omega)$ satisfying the normalization $\int_\Omega \rho\,\rd x=1$. Moreover, $\rho\in W^{2,p}_{\mathrm{loc}}(\Omega)$ for all $1<p<\infty$.
If, in addition, $\bv\in W^{m,\infty}(\Omega)$ for some integer $m \ge 2$, then $\rho\in H^{m}_{\mathrm{loc}}(\Omega)$.
\end{lemma}

\begin{proof}
The function $\rho\in H^1(\Omega)$ with the normalization condition
$\int_\Omega\rho\,\rd x=1$ can be decomposed  as
\[
\rho \;=\; c + u, 
\]
where the constant $c=\frac{1}{|\Omega|}$, and $u\in H^1_\diamond(\Omega)$.
Then the problem \eqref{eq:a_bilinear} is equivalent to: find $u\in H^1_\diamond(\Omega)$ such that
\begin{equation}\label{eq:weak-u}
a(u,\varphi) \;= \,\int_\Omega c\,\bv\cdot\nabla\varphi\,\rd x\triangleq F(\varphi),
\quad\forall\,\varphi\in H^1_\diamond(\Omega).
\end{equation}

Clearly $F$ is a bounded linear functional on $H^1_\diamond(\Omega)$. Indeed, for any $\varphi\in H^1_\diamond(\Omega)$,
\[
|F(\varphi)|
=\Big|\int_\Omega c\,\bv\cdot\nabla\varphi\,\rd x\Big|
\le c\,\|\bv\|_{L^\infty(\Omega)}\,\|\nabla\varphi\|_{L^2(\Omega)}
\le C\,\|\varphi\|_{H^1(\Omega)}.
\]
 Moreover,  for any $w\in H^1_\diamond(\Omega)$, it follows from the G\r{a}rding inequality \eqref{eq:agarding} and the Poincar\'e inequality  that
\[
a(w,w)\ge \frac{D}{2}\|\nabla w\|_{L^2}^2-\frac{\|\bv\|_{L^\infty}^2}{2D}\|w\|_{L^2}^2
\ge \Big(\frac{D}{2}-\frac{\|\bv\|_{L^\infty}^2}{2D}C_P^2\Big)\|\nabla w\|_{L^2}^2= \kappa\,\|\nabla w\|_{L^2}^2,
\]
where $\kappa = \frac{D}{2}-\frac{\|\bv\|_{L^\infty}^2}{2D}C_P^2$.
The condition \( D>\|\bv\|_{L^\infty(\Omega)}\,C_P \) ensures $\kappa>0$.
Since $\|w\|_{H^1}^2=\|w\|_{L^2}^2+\|\nabla w\|_{L^2}^2\le (1+C_P^2)\|\nabla w\|_{L^2}^2$
on $H^1_\diamond(\Omega)$, we obtain
\[
a(w,w)\ge \frac{\kappa}{1+C_P^2}\,\|w\|_{H^1}^2.
\]
Thus the bilinear form $a(\cdot,\cdot)$ is coercive on $H^1_\diamond(\Omega)$. 
Together with its boundedness established in \eqref{eq:abound},
the Lax-Milgram theorem  \cite[Theorem~6.2]{evans2022partial}) yields a unique $u\in H^1_\diamond(\Omega)$ solving \eqref{eq:weak-u}, which also solves \eqref{eq:a_bilinear}. Setting $\rho = c+u$ gives a weak solution of \eqref{eq:ZFN-strong} with $\int_\Omega\rho\,\rd x=1$.

It follows from Proposition \ref{prop:weak} that $\rho$ satisfies \eqref{eq:ZFN-strong} almost everywhere, which can be rewritten as the  non-divergence form:
 \[
-D\Delta\rho + \bv \cdot \nabla\rho + (\nabla \cdot \bv)\rho = 0.
 \]
Standard $L^p$-regularity \cite[Theorems~9.11 and 9.15]{gilbarg1977elliptic} implies that  $\rho$ belongs to $W^{2,p}_{\mathrm{loc}}(\Omega)$ for all $1 < p < \infty$.  Furthermore, if $\bv \in W^{m, \infty}$ for some integer $m\ge2$, then $\nabla\cdot \bv \in W^{m-1, \infty}$. Since $W^{k,\infty}(\Omega)\hookrightarrow C^{k-1,1}(\overline{\Omega})$ for $k\ge1$ by the Sobolev embedding theorem, the coefficients of the elliptic equation are of class $C^{m-2,1}$ in the interior. Applying higher-order interior regularity theory for elliptic equations  \cite[Theorem~8.12]{gilbarg1977elliptic}), we obtain $\rho\in H^{m}_{\mathrm{loc}}(\Omega)$. 
\end{proof}

The condition \( D > |\mathbf{v}|_{L^\infty(\Omega)} C_P \) serves as a sufficient condition to ensure the coercivity of the bilinear form \( a(\cdot, \cdot) \), thereby enabling the direct application of the Lax-Milgram theorem. Physically, this corresponds to a diffusion-dominated regime (or a small Péclet number regime), where the stochastic diffusion is strong enough to counteract the deterministic drift, ensuring a stable stationary distribution. While the well-posedness of the FP equation may hold under weaker assumptions (e.g., via the Fredholm alternative for convection-dominated flows), this condition provides a rigorous and simplified framework sufficient for establishing the stability of the inverse problem.

The normalization condition $\int_\Omega\rho\,\rd x=1$ is consistent with the constraints \eqref{cond_u} for the FP equation. Lemma~\ref{lem:ZFN} guarantees the existence and interior regularity of the stationary density $\rho$.  Consequently, whenever $\rho>0$ on a  subset $\Omega'\subset\Omega$ with $\overline{\Omega'}$ compact in $\Omega$,
and $\rho\in H^m_{\mathrm{loc}}(\Omega)$ with $m>d/2+1$, its score function
\[
\bs = \nabla\log\rho = \frac{\nabla\rho}{\rho}
\]
is well defined on $\Omega'$ and belongs to $H^{m-1}(\Omega')$. In what follows we simply take $\bs$ as a given element of
$W^{1,\infty}(\Omega;\mathbb{R}^d)$ and study the constrained optimization problem associated with $\bs$.  This regularity assumption is natural in the context of elliptic regularity theory: for sufficiently smooth velocity fields \(\bv\) and under a uniform positivity condition on \(\rho\), the stationary FP equation indeed yields a score function with bounded first derivatives.

\subsection{Well-posedness of the constrained optimization}
\label{sec:finalproof}

We now analyze the optimization problem \eqref{eq:true-prob}
\begin{equation}\label{eq:true-prob1}
\min_{\bv}\ J(\bv):=\tfrac12\|\bv\|_{L^2(\Omega)}^2
\quad\text{subject to}\quad \widetilde{\mathcal{N}}(\bs,\bv)=0\ \text{ in }L^2(\Omega).
\end{equation}
where
\(
\widetilde{\mathcal{N}}(\bs,\bv)\;=\;\bs\!\cdot\!\bv\;+\;\nabla\!\cdot\bv\;-\;D\big(|\bs|^2+\nabla\!\cdot\bs\big).
\)
Throughout this subsection, we assume the score \(\bs\) is given with the regularity that can be justified under suitable conditions on the velocity field, as discussed after Lemma~\ref{lem:ZFN}.

\begin{lemma}[Existence and uniqueness of the minimizer]
\label{lem:opt}
Let \(\bs\in L^\infty(\Omega;\mathbb{R}^d)\) with \(\nabla\!\cdot\bs\in L^2(\Omega)\).
Define the linear operator $\mathcal{C}_\bs: H(\mathrm{div};\Omega)\to L^2(\Omega)$:
\[
\mathcal{C}_\bs(\bv)=\bs\!\cdot\!\bv+\nabla\!\cdot\bv,
\]
and let \(g_\bs=D\big(|\bs|^2+\nabla\!\cdot\bs\big)\).
Then the feasible set
\[
\mathcal{V}_{\mathrm{ad}}(\bs):=\{\bv\in H(\mathrm{div};\Omega):\ \mathcal{C}_\bs(\bv)=g_\bs\}
\]
is a nonempty closed affine subspace of \(H(\mathrm{div};\Omega)\).
Consequently, the problem \eqref{eq:true-prob1} admits a unique minimizer
\(\bv^*\in \mathcal{V}_{\mathrm{ad}}(\bs)\).
\end{lemma}

\begin{proof}
The set \(\mathcal{V}_{\mathrm{ad}}(\bs)\) is affine because it is the inverse image of a single element under the linear operator \(\mathcal{C}_\bs\).
It is nonempty since \(\bv_0=D\,\bs\in H(\mathrm{div};\Omega)\) satisfies
\[
\mathcal{C}_\bs(\bv_0)=\bs\!\cdot\!(D\bs)+\nabla\!\cdot(D\bs)=D\big(|\bs|^2+\nabla\!\cdot\bs\big)=g_\bs.
\]
For any \(\bv\in H(\mathrm{div};\Omega)\),
\[
\|\mathcal{C}_\bs(\bv)\|_{L^2}
\le \|\bs\|_{L^\infty}\|\bv\|_{L^2}+\|\nabla\!\cdot\bv\|_{L^2}
\le C\,\|\bv\|_{H(\mathrm{div})},
\]
hence \(\mathcal{C}_\bs\) is continuous and \(\mathcal{V}_{\mathrm{ad}}(\bs)\) is closed.

The objective \(J(\bv)=\tfrac12\|\bv\|_{L^2(\Omega)}^2\) is strongly continuous and strictly  convex on the Hilbert space \(H(\mathrm{div};\Omega)\). Thus it is also weakly lower semicontinuous. The constraint in \eqref{eq:true-prob1} can be written as \(\mathcal{C}_\bs(\bv)=g_\bs\) in \(L^2(\Omega)\) (via the continuous embedding \(H(\mathrm{div};\Omega)\hookrightarrow L^2(\Omega;\mathbb{R}^d)\)).
Let \(\{\bv_n\}\subset\mathcal{V}_{\mathrm{ad}}(\bs)\) be a minimizing sequence. Since \(\{J(\bv_n)\}\) is bounded, \(\{\bv_n\}\) is bounded in \(L^2(\Omega)\); hence, up to a subsequence, \(\bv_n\rightharpoonup \bv^\ast\) weakly in \(L^2(\Omega)\).
Because \(\mathcal{V}_{\mathrm{ad}}(\bs)\) is a closed convex subset of \(H(\mathrm{div};\Omega)\) and \(\mathcal{C}_\bs\) is continuous, the limit \(\bv^\ast\) is feasible, i.e., \(\bv^\ast\in\mathcal{V}_{\mathrm{ad}}(\bs)\).
By weak lower semicontinuity of \(J\),
\[
J(\bv^\ast)\le \liminf_{n\to\infty}J(\bv_n)=\inf_{\bv\in\mathcal{V}_{\mathrm{ad}}(\bs)}J(\bv),
\]
so \(\bv^\ast\) is a minimizer. Uniqueness follows from strict convexity of \(J\) on the convex set \(\mathcal{V}_{\mathrm{ad}}(\bs)\).
\end{proof}

Having established the existence and uniqueness of the minimizer for problem \eqref{eq:true-prob1}, we next analyze the stability of the solution operator \(\bs\mapsto\bv^*\) with respect to perturbations in the score function. 



\begin{lemma}[Continuity of the solution map \(\bs\mapsto\bv^*\)]\label{the:3}
Assume $\bs\in W^{1,\infty}(\Omega;\mathbb{R}^d)$ with $q_{\bs}\ge0$ almost everywhere in $\Omega$, where
\begin{equation}\label{eq:def-q_s}
q_{\bs}=|\bs|^2+\nabla\!\cdot\bs.
\end{equation}
 Let $\bv^*\in H(\mathrm{div};\Omega)$ be the unique minimizer of \eqref{eq:true-prob1}.
Then there exists a neighborhood $U\subset W^{1,\infty}(\Omega;\mathbb{R}^d)$ of $\bs$ such that, for every $\widetilde{\bs}\in U$, if $\widetilde{\bv}^*$ denotes the unique optimizer associated with $\widetilde{\bs}$, then
\[
\|\bv^*-\widetilde{\bv}^*\|_{L^2(\Omega)}
\ \le\ C\,\|\bs-\widetilde{\bs}\|_{W^{1,\infty}(\Omega)},
\]
where  the constant $C>0$ depends monotonically on $D$, $\Omega$, and $\bs$.
\end{lemma}
\begin{proof}
We first characterize the minimizer of the constrained optimization problem \eqref{eq:true-prob1} via the Lagrange multiplier method by
introducing the Lagrangian functional on $H(\mathrm{div};\Omega)\times H^1_\diamond(\Omega)$:
\[
\mathcal{L}(\bv,\lambda;\bs)
=\tfrac12\|\bv\|_{L^2}^2
+\big(\bs\!\cdot\!\bv+\nabla\!\cdot\bv-Dq_\bs,\lambda\big)_{L^2},
\]
where $(\cdot,\cdot)_{L^2}$ denotes the $L^2$-inner product, i.e., $(f,g)_{L^2}:=\int_\Omega f(x)g(x)\,dx$. Since $\bv^\ast$ is a minimizer of \eqref{eq:true-prob1}, the first-order optimality condition implies that the G\^{a}teaux derivative of $\mathcal{L}$ with respect to $\bv$ vanishes at $\bv^\ast$ in every direction $\phi\in H(\mathrm{div};\Omega)$, i.e.,
\begin{equation}\label{eq:KKT-bv}
\delta_\bv\mathcal{L}(\bv^\ast,\lambda)[\phi]
=(\bv^\ast,\phi)_{L^2}
+\big(\bs\!\cdot\!\phi+\nabla\!\cdot\phi,\lambda\big)_{L^2}=0.
\end{equation}

Recall that the FP equation \eqref{eq:steady-pde} is accompanied by the no-flux boundary condition \eqref{eq:bc:simple}, $(D\nabla\rho-\rho\bv)\cdot\bn=0$,  which in terms of the score $\bs=\nabla\log\rho=\nabla \rho/\rho$ reads $(\bv-D\bs)\cdot\bn=0$.
We therefore require that the recovered velocity field satisfy the same condition. Consequently, the admissible directions $\phi$ should preserve this boundary condition, i.e., $\phi\cdot\bn=0$ on $\partial\Omega$. 

For $\lambda\in H^1_\diamond(\Omega)$ and $\phi\in H(\mathrm{div};\Omega)$, we interpret the term $(\lambda,\nabla\!\cdot\phi)_{L^2}$ in \eqref{eq:KKT-bv}  in the distributional sense:
\begin{align}
\label{eq:lambda_term}
  (\lambda,\nabla\!\cdot\phi)_{L^2}
=-(\nabla\lambda,\phi)_{L^2}
+\int_{\partial\Omega}\lambda \phi\!\cdot\!\bn \rd x=-\big(\nabla\lambda,\phi\big)_{L^2}, 
\end{align}
where the boundary integral vanishes due to $\phi\cdot\bn=0$. Substituting this into \eqref{eq:KKT-bv} yields
\[
(\bv^\ast,\phi)_{L^2}
-\big(\nabla\lambda-\lambda\bs,\phi\big)_{L^2}=0
\qquad\forall\,\phi\in H(\mathrm{div};\Omega),
\]
from which we infer the pointwise identity
\begin{equation}\label{eq:vstar-recon}
\bv^\ast=\nabla\lambda-\lambda\bs.
\end{equation}


Substituting $\bv^\ast = \nabla\lambda - \lambda\bs$ into the constraint  
$\bs\!\cdot\!\bv^\ast + \nabla\!\cdot\!\bv^\ast = D(|\bs|^2 + \nabla\!\cdot\bs)$ gives  
\[
-\Delta\lambda + q_\bs\,\lambda = - Dq_\bs \quad \text{in } \Omega.
\]
We supplement this equation with the homogeneous Neumann condition $\partial_n\lambda = 0$ on $\partial\Omega$ and the normalization $\int_\Omega \lambda\,dx = 0$ to obtain a well‑posed problem. Its weak formulation reads: find $\lambda\in H^1_\diamond(\Omega)$ such that
\begin{equation}\label{eq:weak-lambda}
a_\bs(\lambda,\varphi)
=\int_\Omega -Dq_\bs\,\varphi\,\rd x,
\qquad
\forall\,\varphi\in H^1_\diamond(\Omega),
\end{equation}
where $a_\bs$ is the bilinear form defined as
\begin{equation}\label{eq:def-a_s}
a_{\bs}(\lambda,\varphi)=\int_\Omega \nabla\lambda\cdot\nabla\varphi+q_{\bs}\lambda\varphi\,\rd x,
\qquad \lambda,\varphi\in H^1_\diamond(\Omega).
\end{equation}

If $q_{\bs}\geq 0$, then the bilinear form $a_\bs$ is coercive and bounded on
$H^1_\diamond(\Omega)\times H^1_\diamond(\Omega)$. In fact, by Cauchy-Schwarz inequality, for $\lambda,\varphi \in H^1_\diamond(\Omega)$, we have  
\begin{align}
\label{eq:continuity}
    |a_\bs(\lambda,\varphi)|&\leq \|\nabla \lambda\|_{L^2}\|\nabla \varphi\|_{L^2}+|q_{\bs}|{L^\infty} \, \| \lambda\|_{L^2}\| \varphi\|_{L^2} \leq C \| \lambda\|_{H^1}\| \varphi\|_{H^1}, \\
    |a_\bs(\lambda,\lambda)|&=\|\nabla \lambda\|_{L^2}^2+\int_\Omega q_\bs \lambda^2 \rd x\geq \|\nabla \lambda\|_{L^2}^2 \geq \frac{1}{1+C_p^2}\|\lambda\|_{H^1}^2,\label{eq:coercivity}
\end{align}
where $C_p$ is the Poincar\'e constant on $H^1_\diamond(\Omega)$. The right-hand side in \eqref{eq:weak-lambda}
is bounded on $H^1_\diamond(\Omega)$.
Therefore, by the Lax–Milgram theorem, there exists a unique
$\lambda(\bs)\in H^1_\diamond(\Omega)$ solving \eqref{eq:weak-lambda},
and the bound $\|\lambda(\bs)\|_{H^1}\le C\|q_\bs\|_{L^2}$ holds.

Define
\[
F:W^{1,\infty}(\Omega;\mathbb{R}^d)\times H^1_\diamond(\Omega)\to H^{-1}(\Omega),\qquad
F(\widetilde{\bs},\lambda):=-\Delta\lambda+q_{\widetilde{\bs}}\lambda+Dq_{\widetilde{\bs}},
\]
with $q_{\widetilde{\bs}}=|\widetilde{\bs}|^2+\nabla\!\cdot\widetilde{\bs}$.
Then $F$ is continuously Fréchet differentiable. We further define 
\[
D_\lambda F(\bs,\lambda)=-\Delta+q_\bs:H^1_\diamond(\Omega)\to H^{-1}(\Omega),
\]
which is associated with the bilinear form $a_\bs$ by
 $(D_\lambda F(\bs,\lambda)(\lambda),\varphi)_{L^2}=a_{\bs}(\lambda,\varphi)$ for all
$\lambda,\varphi\in H^1_\diamond(\Omega)$.
The coercivity of $a_\bs$ implies that the operator $D_\lambda F(\bs,\lambda)$ is a bounded linear isomorphism $H^1_\diamond(\Omega)\to H^{-1}(\Omega)$.
Hence, by Theorem~\ref{thm:implicitfunction}, there exist a neighborhood
$U\subset W^{1,\infty}(\Omega;\mathbb{R}^d)$ of $\bs$
and a $C^1$ map $\Lambda:U\to H^1_\diamond(\Omega)$ such that
$F(\widetilde{\bs},\Lambda(\widetilde{\bs}))=0$ for all $\widetilde{\bs}\in U$,
and $\Lambda$ is locally Lipschitz in the $H^1$-norm:
\[
\|\Lambda(\widetilde{\bs})-\Lambda(\bs)\|_{H^1}
\le C\,\|\widetilde{\bs}-\bs\|_{W^{1,\infty}}.
\]

From \eqref{eq:vstar-recon}, for $\widetilde{\bs}\in U$ we have
\[
\bv^\ast(\widetilde{\bs})-\bv^\ast(\bs)
=\nabla(\Lambda(\widetilde{\bs})-\Lambda(\bs))
-\Lambda(\widetilde{\bs})\,\widetilde{\bs}
+\Lambda(\bs)\,\bs.
\]
Taking the $L^2$-norm and using the uniform boundedness of
$\Lambda(\cdot)$ on $U$ gives
\[
\|\bv^\ast(\widetilde{\bs})-\bv^\ast(\bs)\|_{L^2}
\le C\,\|\Lambda(\widetilde{\bs})-\Lambda(\bs)\|_{H^1}
+C\,\|\Lambda(\bs)\|_{L^2}\|\widetilde{\bs}-\bs\|_{L^\infty}
\le C\,\|\widetilde{\bs}-\bs\|_{W^{1,\infty}}.
\]
This proves the claimed local Lipschitz continuity.
\end{proof}

\begin{rem}
The condition $q_{\bs}\ge0$ is sufficient to ensure coercivity of the bilinear form $a_{\bs}$ used in the proof. This can be relaxed to the weaker condition $\|q_{\bs}^-\|_{L^\infty(\Omega)} < C_P^{-2}$, where 
$q_{\bs}^-(x)=\max\{-q_{\bs}(x),0\}$ and $C_P$ is the Poincar\'e constant 
for $H^1_\diamond(\Omega)$. Indeed, under this condition we have
\begin{align*}
    |a_\bs(\lambda,\lambda)|&\geq\|\nabla \lambda\|_{L^2}^2-\int_\Omega q_\bs^- \lambda^2 \rd x\geq (1-C_P^2\|q_\bs\|_{L^\infty}) \|\nabla \lambda\|_{L^2}^2
   \geq \frac{1-C_P^2\|q_\bs\|_{L^\infty}}{1+C_p^2}\|\lambda\|_{H^1}^2,
\end{align*}
so that $a_\bs$ remains coercive. Hence, all conclusions of the theorem remain valid under this weaker condition.
\end{rem}


With the lemmas established above, we are ready to prove Theorem \ref{them:unique}.
\begin{proof}[Proof of Theorem \ref{them:unique}]
Under suitable conditions on the velocity field,
Lemma \ref{lem:ZFN} justifies the regularity of the invariant density and thus the regularity of the score function. Given such a score function $\bs$,
Lemma~\ref{lem:opt} establishes the existence and uniqueness of a minimizer \(\bv^*\) for problem \eqref{eq:true-prob1}. Lemma~\ref{the:3} shows that the solution map \(\bs\mapsto\bv^*\) is locally Lipschitz continuous from \(W^{1,\infty}(\Omega;\mathbb{R}^d)\) into \(L^2(\Omega;\mathbb{R}^d)\), which yields the stability assertion. This completes the proof of the theorem.
\end{proof}

\end{document}